\documentclass[11pt]{amsart}
\usepackage{graphicx}
\usepackage{amssymb}
\usepackage{amsmath}
\usepackage{amsfonts}
\usepackage{color}
\usepackage[margin=1in]{geometry}
\usepackage[T1]{fontenc}
\usepackage[utf8]{inputenc}
\usepackage[font=small,labelfont=bf]{caption}
\usepackage{pb-diagram}
\usepackage[cmtip, arrow]{xy}
\usepackage{pb-diagram, pb-xy}
\usepackage{xy}
\usepackage{soul}

\xyoption{all}
\newtheorem{theorem}{Theorem}
\theoremstyle{plain}

\newtheorem{corollary}{Corollary}

\newtheorem{definition}{Definition}
\newtheorem{example}{Example}

\newtheorem{lemma}{Lemma}

\newtheorem{proposition}{Proposition}
\newtheorem{remark}{Remark}

\numberwithin{equation}{section}

\begin{document}
\title[Epis of ordered algebras]{On Epimorphisms of Ordered Algebras}
\author{Nasir Sohail}
\address{Department of Mathematics, Wilfrid Laurier University, Waterloo,
ON, Canada}
\email{nsohail@wlu.ca}
\thanks{}
\author{Boza Tasi\' c}
\curraddr{Ted Rogers School of Management, Ryerson University, Toronto, ON,
Canada}
\email{btasic@ryerson.ca}
\date{November 6, 2017}
\subjclass{}
\keywords{Ordered algebra, Variety, Epimorphism, Amalgamation}

\begin{abstract}
We prove that epimorphisms are surjective in certain categories of ordered 
$\mathcal{F}$-algebras. It then turns out that epimorphisms are also
surjective in the category of all (unordered) algebras of type $\mathcal{F}$.
\end{abstract}

\maketitle

\section{Motivation}

Various ordered algebras appeared in different contexts in mathematics, mostly during the
second half of the previous century. The monograph \cite{Fuchs} by Fuchs, written in the early 1960's, gave 
an outline of the theory of ordered groups, rings, fields and semigroups. Around the same time, 
universal algebra and lattice theory had also started to flourish. 
The initial account of ordered universal algebras, to the best knowledge of the
authors, was provided by Bloom in \cite{Bloom}. 

It has been recently shown in \cite{NS} that epimorphisms (briefly, epis) 
are preserved by the forgetful functor from the category of partially ordered monoids (briefly, pomonoids) 
to the category of (unordered) monoids. The aim of the present article is to pursue this line of research in the context 
of ordered universal algebras. 
Accordingly, for a type $\mathcal{F}$ of universal algebras, we prove that epis are surjective in
certain categories of ordered $\mathcal{F}$-algebras, as well as in their underlying category of all (unordered) 
$\mathcal{F}$-algebras. We begin by introducing basic notions in Section \ref{SecPrelim}. This is followed by a
discussion, in Section \ref{SecOTA}, about the ordered term algebras. 
In Section \ref{SecAmal} we introduce categorical concepts and construct objects that we need.
Finally, we employ a kind of term re-writing technique to prove our main result in Section \ref{SecEpis}.

\section{Preliminaries}\label{SecPrelim}

Basic universal algebraic and category theoretic notions and definitions are adopted from \cite{Burris} and \cite{MacLane} respectively. 
We denote the \emph{type} of a universal algebra by $\mathcal{F}$ and for every 
$k\in \mathbb{N}\cup \{0\}$ we denote by $\mathcal{F}_{k}$ the set of all $k$-ary operation symbols.

\begin{definition}[see \protect\cite{Bloom}]
\upshape\label{Ordered algebra}Given a type $\mathcal{F}$, an \emph{ordered $%
\mathcal{F}$-algebra} is a triple $\mathbf{A}=(A,\mathcal{F}^{\mathbf{A}%
},\leq _{A})$, where $(A,\mathcal{F}^{\mathbf{A}})$ is an $\mathcal{F}$%
-algebra and $(A,\leq _{A})$ is a partially ordered set (briefly, poset),
such that every $f^{\mathbf{A}}\in \mathcal{F}^{\mathbf{A}}$ is a monotone
function, i.e., if $f\in \mathcal{F}_{k}$ and $a_{1},b_{1},\ldots
,\,a_{k},b_{k}\in A$, then 
\begin{equation}
\left( a_{1}\leq _{A}b_{1}\wedge \cdots \wedge a_{k}\leq _{A}b_{k}\right)
\Rightarrow f^{\mathbf{A}}\left( a_{1},\ldots ,a_{k}\right) \leq _{A}f^{%
\mathbf{A}}\left( b_{1},\ldots ,b_{k}\right) .  \label{1}
\end{equation}
\end{definition}

\noindent When condition (\ref{1}) is satisfied, we say that $\leq _{A}$ is 
\textit{compatible with the operations} in $\mathcal{F}^{\mathbf{A}}$. Every 
$\mathcal{F}$-algebra $(A,\mathcal{F}^{\mathbf{A}})$ can be made into an
ordered $\mathcal{F}$-algebra by endowing it with the trivial order $=$. If
there is no ambiguity in dropping the index, we shall use $\leq $ instead of 
$\leq _{A}$.

Let $\mathbf{A}=(A, \mathcal{F}^{\mathbf{A}}, \leq_{A})$ and $\mathbf{B}=(B, 
\mathcal{F}^{\mathbf{B}}, \leq _{B})$ be ordered $\mathcal{F}$-algebras. We
say that $\mathbf{B}$ is a \textit{subalgebra} of $\mathbf{A}$ if

\begin{enumerate}
\item[(i)] $(B, \mathcal{F}^{\mathbf{B}})$ is a subalgebra of $(A, \mathcal{F%
}^{\mathbf{A}})$, and

\item[(ii)] $\leq _{B}\;=\;\leq _{A}\cap ~(B\times B)$.
\end{enumerate}

\noindent A \textit{homomorphism} $f:\mathbf{A}\longrightarrow \mathbf{B}$
of ordered $\mathcal{F}$-algebras is a monotone function $f:(A,\leq
_{A})\longrightarrow (B,\leq _{B})$ that is also a homomorphism of the
underlying $\mathcal{F}$-algebras. We call $f$ an \textit{order-embedding}
if it also reflects the order, i.e., $f(x)\leq _{B}f(y)\Rightarrow x\leq_{A}y$, for all $x,y\in A$. 
Every order-embedding is necessarily injective.
The \textit{product} $\prod_{i\in I}\mathbf{A}_{i}$ of a family $\mathbf{A}
_{i}=(A_{i},\mathcal{F}^{\mathbf{A_{i}}},\leq _{A_{i}})$ of ordered $\mathcal{F}$-algebras
is obtained by defining component-wise operations and
order on $\prod_{i\in I}A_{i}$. A class $\mathcal{K}$ of ordered $\mathcal{F}$-algebras 
is called a \textit{variety} if it is closed under homomorphic
images, subalgebras and products. According to the Birkhoff-type 
characterization of varieties given in \cite{Bloom}, $\mathcal{K}$ is a variety iff it consists precisely
of all ordered $\mathcal{F}$-algebras satisfying a given set of inequalities 
$s\leq t$, where $s$ and $t$ are terms of type $\mathcal{F}$. Naturally, every variety of ordered $\mathcal{F}$
-algebras gives rise to a category; indeed a $2$-category enriched over the category of posets. 
\textit{Epimorphisms (monomorphisms)} in varieties of ordered algebras are the right (left) cancelative 
homomorphisms.  It was observed in 
\cite{Valdis Nasir} that 
monomorphisms (isomorphisms) in the varieties of ordered 
$\mathcal{F}$-algebras are precisely the injective homomorphisms (surjective order-embeddings).

Let $\mathrm{F}$ denote the forgetful functor from a category $\mathfrak{C}$
of ordered $\mathcal{F}$-algebras to its underlaying category of unordered 
$\mathcal{F}$-algebras. Then $g\in  Hom(\mathfrak{C})$ is clearly an
epimorphism if $\mathrm{F}(g)$ is such. So, epimorphisms in the varieties of ordered 
$\mathcal{F}$-algebras need not be surjective, since they are not necessarily such
in the varieties of unordered $\mathcal{F}$-algebras. In this
article we prove that epimorphisms are surjective in the varieties of ordered 
$\mathcal{F}$-algebras defined by the inequalities $c\leq d$, where $c,d\in 
\mathcal{F}_{0}$. We shall denote these varieties by $\mathcal{F}$-\textbf{Oalg}$_{\leq }^{0}$.

Recall that a reflexive and transitive relation on a set is called a 
\textit{quasiorder}. Given an ordered $\mathcal{F}$-algebra 
$\mathbf{A}=(A,\mathcal{F}^{\mathbf{A}},\leq )$, a quasiorder $\sigma $ on $A$ is called a 
\textit{compatible quasiorder} on $\mathbf{A}$ if it is compatible with the
operations in $\mathcal{F}^{\mathbf{A}}$ and extends the order on $A$, i.e., 
$\leq ~\subseteq ~\sigma \,$. For a homomorphism $f:\mathbf{A}
\longrightarrow \mathbf{B}$ of ordered $\mathcal{F}$-algebras we define 
$\overset{\longrightarrow }{\ker }f$, the \textit{directed kernel} of $f$, by
\begin{equation*}
\overset{\longrightarrow }{\ker }f=\left\{ (a,b)\in A\times A:f(a)\leq
f(b)\right\} \text{.}
\end{equation*}
The relation $\overset{\longrightarrow }{\ker }f$ is a compatible quasiorder
on $\mathbf{A}$. In fact, every compatible quasiorder on $\mathbf{A}$ turns
out to be the directed kernel of some homomorphism 
$f:\mathbf{A}\longrightarrow \mathbf{B}$, see \cite{Gabor and Attila}. Given an ordered 
$\mathcal{F}$-algebra $\mathbf{A}=(A,\mathcal{F}^{\mathbf{A}},\leq )$ and a
congruence $\theta $ of the algebra $(A,\mathcal{F}^{\mathbf{A}})$, we
define the relation $\underset{\theta }{\leq }\,\subseteq A^{2}$ by 
\begin{equation*}
a\underset{\theta }{\leq }b\Leftrightarrow (\exists n\in \mathbb{N})(\exists
a_{1},b_{1},\cdots ,a_{n},b_{n}\in A)(a\leq a_{1}\theta b_{1}\leq \cdots
\leq a_{n}\theta b_{n}\leq b).
\end{equation*}
The relation $\underset{\theta }{\leq }$ is also a compatible quasiorder on 
$\mathbf{A}$.

An \textit{order-congruence} of an ordered $\mathcal{F}$-algebra 
$\mathbf{A}=(A,\mathcal{F}^{\mathbf{A}},\leq )$ is a congruence $\theta $ of the
algebra $(A,\mathcal{F}^{\mathbf{A}})$ satisfying the following condition: 
\begin{equation}
(\forall a,b\in A)\,(a\underset{\theta }{\leq }b\underset{\theta }{\leq }
a\Rightarrow a\theta b)\,\text{.}  \label{ccc}
\end{equation}
Condition (\ref{ccc}) is also known as the \textit{closed chain condition}.
If $\sigma $ is a compatible quasiorder on $\mathbf{A}$, then $\sigma \cap
\sigma ^{-1}$ is an order-congruence of $\mathbf{A}$. For example,
\begin{equation*}
\ker f=\left( \overset{\longrightarrow }{\ker }f\right) \cap \left( \overset{
\longrightarrow }{\ker }f\right) ^{-1}\text{.}
\end{equation*}

The \textit{regular quotient} (see \cite{Valdis Nasir}) of an ordered
algebra $\mathbf{A}=(A,\mathcal{F}^{\mathbf{A}},\leq )$, by an
order-congruence $\theta $, is the ordered algebra 
\begin{equation*}
\mathbf{A}/\theta =\left( A/\theta ,\mathcal{F}^{\mathbf{A}/\theta },\leq
_{A/\theta }\right) \text{,}
\end{equation*}
such that $(A/\theta ,\mathcal{F}^{\mathbf{A}/\theta })$ is the usual
algebraic quotient, and the order $\leq _{A/\theta }$ is defined by 
\begin{equation*}
\lbrack a]\leq _{A/\theta }[b]\,\Leftrightarrow \,a\underset{\theta }{\leq }b\,\text{.}
\end{equation*}
One can easily observe that $\leq _{A/\theta }$ is the coarsest compatible
order on $A/\theta $ that makes the canonical homomorphism $\theta
^{\natural }:\mathbf{A}\longrightarrow \mathbf{A}/\theta $ monotone.
Proof of the next theorem is straightforward, and is omitted.

\begin{theorem}
\label{Homomorphism} Let $f:\mathbf{A}\longrightarrow \mathbf{B}$ be a
homomorphism of ordered algebras, and let $\theta$ be an order-congruence on 
$\mathbf{A}$ such that $\underset{\theta }{\leq }\,\,\subseteq\, \overset{%
\longrightarrow}{\ker }f$. Then there exists a unique homomorphism $g:%
\mathbf{A}/\theta \longrightarrow \mathbf{B}$, such that $g\circ \theta
^{\natural }=f$.
\end{theorem}

Given an ordered algebra $\mathbf{A}=(A,\mathcal{F}^{\mathbf{A}},\leq )$ and
a compatible quasiorder $\sigma $, we define the \textit{non-regular quotient%
} of $\mathbf{A}$ by $\sigma $ to be the ordered algebra%
\begin{equation*}
\mathbf{A}/\sigma =\left( A/\left( \sigma \cap \sigma ^{-1}\right) ,\mathcal{%
F}^{\mathbf{A}/\left( \sigma \cap \sigma ^{-1}\right) },\preccurlyeq \right) 
\text{,}
\end{equation*}%
where $\left( A/\left( \sigma \cap \sigma ^{-1}\right) ,\mathcal{F}^{\mathbf{%
A}/\left( \sigma \cap \sigma ^{-1}\right) }\right) $ is the usual algebraic
quotient and the order $\preccurlyeq $ is defined by 
\begin{equation*}
\lbrack a]\preccurlyeq \lbrack b]\,\Leftrightarrow \,a\sigma b\,\text{.}
\end{equation*}%
Note that $\mathbf{A}/\left( \sigma \cap \sigma ^{-1}\right) $ denotes the
regular quotient algebra 
\begin{equation*}
\left( A/\left( \sigma \cap \sigma ^{-1}\right) ,\mathcal{F}^{\mathbf{A}%
/\left( \sigma \cap \sigma ^{-1}\right) },\leq _{A/\left( \sigma \cap \sigma
^{-1}\right) }\right) \,\text{.}
\end{equation*}%
Both the quotients $\mathbf{A}/\sigma $ and $\mathbf{A}/\left( \sigma \cap
\sigma ^{-1}\right) $ have the same universe and operations, however $%
\mathbf{A}/\sigma $ is \textquotedblleft more ordered\textquotedblleft\ than 
$\mathbf{A}/\left( \sigma \cap \sigma ^{-1}\right) $ in the sense that $%
\leq _{A/\left( \sigma \cap \sigma ^{-1}\right) }$ is contained in $%
\preccurlyeq $. Every variety of ordered algebras is closed under both
regular and non-regular quotients.

\section{Ordered Term Algebra}\label{SecOTA}

Let $\mathcal{F}$ be a type of algebras and let $X$ be a set such that 
$X\cap \mathcal{F}=\emptyset $. We call $X$ the set of \textit{variables} and
we assume that either $X$ or $\mathcal{F}_{0}$ is nonempty. Whenever $X$
($\mathcal{F}_{0}$) is nonempty we also assume that it is
equipped with a partial order which we denote by $\leq _{X}$ (
$\leq _{\mathcal{F}_{0}}$). A \textit{word} on $X\cup \mathcal{F}$ is a
nonempty finite sequence of elements of $X\cup \mathcal{F}$. We concatenate
sequences by simple juxtaposition.

\begin{definition}
\upshape Given a type $\mathcal{F}$ of algebras and a set $X$ of variables
we define, by recursion on $n$, the sets $T_{n}$ of words on 
$X\cup \mathcal{F}$ by 
\begin{eqnarray*}
T_{0} &=&\{w|w\in X\cup \mathcal{F}_{0}\} \\
T_{n+1} &=&T_{n}\cup \{fs_{1}s_{2}\dots s_{k}|f\in \mathcal{F}
_{k},s_{1},s_{2},\dots ,s_{k}\in T_{n}\}.
\end{eqnarray*}
We, then, define $T_{\mathcal{F}}(X)=\cup _{n\in \mathbb{N}}T_{n}$, called
the set of \textit{terms of type} $\mathcal{F}$ \textit{over} $X$.
\end{definition}

\noindent Although, in algebra we write more often $f(s_{1},s_{2},\dots
,s_{k})$ instead of $fs_{1}s_{2}\dots s_{k}$, we shall be using both 
notations interchangeably. The set $T_{\mathcal{F}}(X)$ of all terms of type 
$\mathcal{F}$ over the set of variables $X$ is the universe of the \textit{term algebra} 
$\mathbf{T}(X)=\left( T_{\mathcal{F}}(X),\mathcal{F}^{\mathbf{T}(X)}\right)$, where for 
$f\in \mathcal{F}_{k}$ the operation $f^{\mathbf{T}(X)}$ is defined by 
\begin{equation}
f^{\mathbf{T}(X)}\left( t_{1},\ldots ,t_{k}\right) =f(t_{1},\ldots ,t_{k}), 
\label{talgop}
\end{equation}
$t_{1},\ldots ,t_{k}\in T_{\mathcal{F}}(X)$. A \textit{constant free term}
is a term of type $\mathcal{F}$ over $X$ that contains no constant symbols.
We use $t(x_{1},x_{2},\dots ,x_{n})$ to denote a term whose variables are
among $x_{1},x_{2},\dots ,x_{n}$. 

Recall that a\textit{\ }\textit{tree} comprises a set of points, called \textit{%
nodes}, and a set of lines, called \textit{edges}. The edges connect the
nodes so that there is exactly one path between any two different nodes. A
rooted tree is a tree in which a node is designated as the root. Formally,
we define \textit{rooted trees} inductively as follows.

\begin{enumerate}
\item[(i)] A single node $n$ is a tree. In this case $n$ is the root of this
one-node tree.

\item[(ii)] If $T_{1},T_{2},\dots ,T_{k}$ are trees with the roots $%
c_{1},c_{2},\dots ,c_{k}$ respectively and $r$ is a new node, then we form a
new tree $T$ from $r$ and $T_{1},T_{2},\dots ,T_{k}$ in the following way.

\begin{enumerate}
\item $r$ becomes the root of $T$.

\item We add an edge from $r$ to each of the nodes $c_{1},c_{2},\dots ,c_{k}.$%
\end{enumerate}
\end{enumerate}

\noindent From now on, by a tree we shall always mean a rooted tree. A node $%
a$ in a tree is called the \textit{parent} of a node $b$ if $a$ is adjecent
to $b$ on the path between $b$ and the root. In this case, we also call $b$
a \textit{child} of $a$. A \textit{leaf} is a node of a tree that has no
children. A \textit{labelled tree} is a tree with a label attached to each
node. An \textit{ordered labeled tree} is a labeled tree in which all
children of each non-leaf node are ordered linearly first to last
(left to right). Let $T$ be an ordered labeled tree. By removing all 
leaf labels from $T$ we get its \textit{skeleton}, which we shall denote by $%
skelt(T)$. Let us further agree that the trees we consider will
all be ordered labeled trees.

\begin{definition}
\upshape Given a term $t\in T_{\mathcal{F}}(X)$, we define the \textit{tree
of }$\,t$, denoted by $Tree(t)$ as follows.

\begin{enumerate}
\item If $t\in X\cup \mathcal{F}_{0}$, then $Tree(t)$ is just one-node: 
\begin{figure}[h]
\begin{equation*}
\xymatrix{\overset{t}{\bullet} }
\end{equation*}
\end{figure}

\item $Tree(fs_{1}s_{2}\dots s_{k})$ is depicted in Figure \ref{Tree}.
\begin{figure}[h]
\begin{equation*}
\xymatrix{ \overset{Tree(s_1)}{\bullet} \ar@{-}[drr] &
\overset{Tree(s_2)}{\bullet} \ar@{-}[dr] & \cdots
&\overset{Tree(s_{k-1})}{\bullet} \ar@{-}[dl] & \overset{Tree(s_k)}{\bullet}
\ar@{-}[dll] \\ & & \underset{f}{\bullet} & & }
\end{equation*}
\captionof{figure}{} \label{Tree}
\end{figure}
\end{enumerate}
\end{definition}

\noindent It is easy to see that the tree of a term always satisfies the following
conditions:

\begin{enumerate}
\item Each leaf is labelled by an element from $X\cup \mathcal{F}_{0}.$

\item If $n$ is a non-leaf node with $k$ children then $n$ is labelled
by an operation symbol $f\in \mathcal{F}_{k}$.
\end{enumerate}

\noindent Conversely, every tree satisfying (1) and (2) is the tree of some
term $t\in T_{\mathcal{F}}(X)$. The correspondence between isomorphism 
types of such trees and their terms is a bijection.

\begin{definition}
\upshape Let $t\in T_{\mathcal{F}}(X)$, such that $Tree(t)$ has $n$ leaves.
We define a function $leaf(t):\{1,\ldots ,n\}\longrightarrow X\cup \mathcal{F%
}_{0}$ by 
\begin{equation*}
leaf(t)(i)=l_{i},
\end{equation*}%
where $l_{i}$ is the $i^{th}$ leaf label of $Tree(t)$.
\end{definition}

\begin{example}
\upshape If $t=fgx_{2}x_{1}cfx_{1}x_{4}$ where $g\in \mathcal{F}_{3}$, $f\in 
\mathcal{F}_{2}$ and $c\in \mathcal{F}_{0}$, then $Tree(t)$ has five leaves
as shown in Figure \ref{treet}. We have $leaf(t)(1)=x_{2}$, $%
leaf(t)(2)=leaf(t)(4)=x_{1}$, $leaf(t)(3)=c$ and $leaf(t)(5)=x_{4}$. The
skeleton of $Tree(t)$ is given in Figure \ref{skt}.%
\end{example}
\vspace{0.5cm} 
\begin{minipage}{\textwidth}
\resizebox{2in}{!}{
\begin{minipage}[b]{0.35\textwidth}
    \centering
\begin{equation*} 
\xymatrix{
\overset{x_2}{\bullet} \ar@{-}[dr]   & \overset{x_1}{\bullet} \ar@{-}[d] & \overset{c}{\bullet} \ar@{-}[dl] &  &
\overset{x_1}{\bullet} \ar@{-}[dr]   &  &  \overset{x_4}{\bullet} \ar@{-}[dl]  \\
  & \underset{g}{\bullet} \ar@{-}[drr] &  &                                 & &  \underset{f}{\bullet} \ar@{-}[dll] & \\ 
  &                                                 &  &  \underset{f}{\bullet} &  &                                               &  }
\end{equation*}
  \captionof{figure}{$Tree(t)$} \label{treet}
 \end{minipage}}
\hspace{1in}
\resizebox{2in}{!}{
\begin{minipage}[b]{0.35\textwidth}
    \centering
\begin{equation*} 
\xymatrix{
\overset{}{\bullet} \ar@{-}[dr]   & \overset{}{\bullet} \ar@{-}[d] & \overset{}{\bullet} \ar@{-}[dl] &  &
\overset{}{\bullet} \ar@{-}[dr]   &  &  \overset{}{\bullet} \ar@{-}[dl]  \\
  & \underset{g}{\bullet} \ar@{-}[drr] &  &                                 & &  \underset{f}{\bullet} \ar@{-}[dll] & \\ 
  &                                                 &  &  \underset{f}{\bullet} &  &                                               &  }
\end{equation*}
\captionof{figure}{$skelt(t)$} \label{skt}
 \end{minipage}}
\end{minipage}

\vspace{0.5cm}

\noindent Note that the tree of $s=fgx_{2}x_{1}x_{1}fx_{4}c$ has the same
skeleton as $Tree(t)$.

\begin{definition}
\upshape Given $t\in T_{\mathcal{F}}(X)$, we denote by $var(t)$ the
sequence of variables of $t$, written in the same order as they appear on the leaves of $Tree(t)$.
\end{definition}

If $var(t)=(x_{i_{1}},\ldots ,x_{i_{k}})$, we shall write $t[x_{i_{1}},\ldots ,x_{i_{k}}]$ to specify the variables that are
explicitly occurring in $t$ and their order of occurrence. If $t$ is a term built from constants only, we let 
$var(t)=({\tiny \ })$ and we write $t[{\small \ }]$. This notation will be used when an additional
precision about the variables explicitly occurring in $t$ and their order is needed.

For the tree given in Figure \ref{treet}, we have $var(t)=(x_{2},x_{1},x_{1},x_{4})$. We may also denote this
tree by $t[x_{2},x_{1},x_{1},x_{4}]$. 
Observe that $var(r)=var(t)$, for $r(x_{1},x_{2},x_{3},x_{4})=gcfx_{2}x_{1}fx_{1}x_{4}$, however 
$skelt(r)\not=skelt(t)$. 

From now on, we fix a countably infinite sequence $z_{1},z_{2},\dots,z_{n},\dots $ of formal variables, 
and we let $V=\{z_{n}|n\in \mathbb{N}\}.$ We shall also consider the terms of type $\mathcal{F}$ over $V$.

\begin{definition} \label{regterm}
\upshape A term $t\in T_{\mathcal{F}}(V)$ is called \textit{regular} if for some $n\geq 1$, 
$var(t)=(z_{1},z_{2},\dots ,z_{n})$. We call $n$ the \textit{arity} of the regular term 
$t$ and write $ar(t)=n$. The set of all regular terms of type $\mathcal{F}$ is denoted by 
$T_{\mathcal{F}}^{reg}(V)$.
\end{definition}

\noindent The following two lemmas provide some useful facts about regular terms.

\begin{lemma}
\label{t=regt} Let $t\in T_{\mathcal{F}}(X)$ be such that $Tree(t)$ has $n$ leafs and 
$var(t)=(x_{i_{1}},\ldots ,x_{i_{k}})$, for some $k\leq n$. Then there exist a constant free
term $\overline{t}(z_{1}, \ldots, z_{n})\in T_{\mathcal{F}}^{reg}(V)$ of arity $n$, such that

\begin{itemize}
\item[(i)] $t[x_{i_{1}},\ldots ,x_{i_{k}}]=\overline{t}(\overline{x}_1, 
\overline{x}_2, \dots, \overline{x}_n)$ where $\overline{x}_i\in
\{x_{i_{1}},\ldots ,x_{i_{k}}\}\cup \mathcal{F}_{0}$ for $1\leq i\leq n$,

\item[(ii)] $o(t)=o(\overline{t})$.
\end{itemize}
\end{lemma} 

\begin{proof} By induction on the number $o(t)$ of operation symbols in $t$.
\end{proof}

\begin{lemma}
\label{RW}Let $\mathbf{T}(X)$ be the term algebra of the type $\mathcal{F}$
over $X$. Let $t_{1}(z_{1},\dots ,z_{n}), \,t_{2}(z_{1},\dots ,z_{m})\in T_{\mathcal{F}}^{reg}(V)$ be
constant free, such that $var(t_{1})=(z_{1},\dots ,z_{n})$, 
$var(t_{2})=(z_{1},\dots ,z_{m})$. If 
\begin{equation*}
t_{1}(a_{1},\ldots ,a_{n})=t_{2}(b_{1},\ldots ,b_{m})\text{,}
\label{t(a)=t(b)}
\end{equation*}
for some $a_{1},\ldots ,a_{n},b_{1},\ldots ,b_{m}\in X\cup \mathcal{F}_{0}$,
then $n=m$ and we have

\begin{enumerate}
\item[(i)] $t_{1}(z_{1},\dots ,z_{n})=t_{2}(z_{1},\dots ,z_{m})$,

\item[(ii)] $(a_{1},\ldots ,a_{n})=(b_{1},\ldots ,b_{m})$.
\end{enumerate}
\end{lemma}

\begin{proof}
Straightforward. 
\end{proof}

\begin{definition}
\upshape Let $\mathbf{A}$ be an $\mathcal{F}$-algebra, and let $X\subseteq A$. The 
\textit{X-translations} of $\mathbf{A}$ are the unary polynomials of $\mathbf{A}$ of the form 
\begin{equation*}
p(u)=t(a_1, \dots, a_{i-1}, u, a_{i+1}, \ldots, a_n)\,\, \mbox{for all}%
\,\,u\in A,
\end{equation*}
where $t\in T_{\mathcal{F}}^{reg}(V)$ is of arity $n$ and $a_1, \dots, a_{i-1}, a_{i+1},
\ldots, a_n\in X$.
\end{definition}

A \textit{translation} of $\mathbf{A}$ is any $A$-translation of $\mathbf{A}.$
The $X$-translations induced by the constant free regular terms will be
called $X$-\textit{cfr translations}. Clearly, if $\mathbf{A}$ is an 
$\mathcal{F}$ algebra generated by $X$, then the
translations of $\mathbf{A}$ coincide with the $X$-translations of $\mathbf{A}$. 
Lemma \ref{t=regt} further implies that the $X$-translations of $\mathbf{A}$ 
coincide with the $X\cup \mathcal{F}_{0}^{\mathbf{A}}$ - cfr translations of $\mathbf{A}$.  
Let $\mathbf{A}=(A, \mathcal{F}^{\mathbf{A}}, \leq)$ be an ordered algebra
generated by $X\subseteq A$, and let $\mathsf{H}$ be a binary relation on $A$%
. We define a new relation $\overset{\mathsf{H}}{\longrightarrow }%
\,\subseteq\, A^{2}$ by 
\begin{equation}
\overset{\mathsf{H}}{\longrightarrow }\,=\,\{(p(u), p(v))\, |\, (u, v)\in
H\,\, \mbox{and}\,\, p\,\, \mbox{is an}\,\, X\cup \mathcal{F}_{0}^{\mathbf{A}} -\mbox{cfr translation}\}.
\end{equation}
If the generating set $X$ is not mentioned, $p$ will be assumed to
denote an $A$-cfr~translation. The following lemma gives a useful description of the compatible quasiorder
generated by $\mathsf{H}$ which we shall denote by $\Sigma _{\mathsf{H}}$.

\begin{lemma}[Cf. Lemma 4 of \protect\cite{Valdis Lauri Nasir 2}]
\label{Induced} Let $\mathbf{A}=(A, \mathcal{F}^{\mathbf{A}}, \leq)$
be an ordered algebra and $\mathsf{H}\subseteq A^{2}$. For $c,\,c^{\prime
}\in A$ we have $c\,\Sigma _{\mathsf{H}}\,c^{\prime }$ iff either $c\,\leq\,
c^{\prime }$ or there exists a scheme 
\begin{equation*}
c\leq p_{1}(a_{1})\overset{\mathsf{H}}{ \longrightarrow }p_{1}(a_{1}^{\prime
})\leq p_{2}(a_{2})\overset{\mathsf{H}}{ \longrightarrow }%
p_{2}(a_{2}^{\prime })\leq \cdots \leq p_{n}(a_{n})\overset{\mathsf{H}}{%
\longrightarrow } p_{n}(a_{n}^{\prime })\leq c^{\prime } \text{.}
\end{equation*}
\end{lemma}

\noindent Because $(a,b)\in \mathsf{H}$ implies that $a\,\Sigma _{\mathsf{H}%
}\,b$, we have $\left[ a\right]\preccurlyeq \left[ b\right]$ in $\mathbf{A}%
/\Sigma _{\mathsf{H}}$ whenever $(a,b)\in \mathsf{H}$. The following
proposition follows easily from Lemma \ref{Induced}.

\begin{proposition}
Let $\theta$ be an order-congruence on $\mathbf{A}=(A, \mathcal{F}^{%
\mathbf{A}}, \leq)$ and $\mathsf{H}\subseteq A^{2}$. If for every $a, b \in
A $, $(a,b)\in \mathsf{H}\Rightarrow a\underset{\theta }{\leq }b$, then $%
\Sigma _{\mathsf{H}}\cap \Sigma _{\mathsf{H}}^{-1}\,\subseteq \,\theta$.
\end{proposition}

Given $\mathbf{A}=(A, \mathcal{F}^{\mathbf{A}}, \leq)$ and $\mathsf{H}%
\subseteq A^{2}$ one can also consider the order-congruence $\Theta _{%
\mathsf{H}}$ on $\mathbf{A}$ generated by $\mathsf{H}$. The following lemma
gives a practical description of $\underset{\Theta _{\mathsf{H}}}{\leq }$.

\begin{lemma}[Cf. Lemma 1.2 of \protect\cite{Valdis NS Lauri 1}]
\label{Generated}Let $\mathbf{A}=(A, \mathcal{F}^{\mathbf{A}}, \leq)$
be an ordered algebra and $\mathsf{H}\subseteq A^{2}$. For $%
c,\,c^{\prime}\in A$, we have $c\underset{\Theta _{\mathsf{H}}}\leq
c^{\prime}$ iff either $c\,\leq\, c^{\prime}$ or there exists a scheme 
\begin{equation*}
c\leq p_{1}(a_{1})\overset{\mathsf{H}\cup \mathsf{H}^{-1}}{\longrightarrow }%
p_{1}(a_{1}^{\prime }) \leq p_{2}(a_{2})\overset{\mathsf{H}\cup \mathsf{H}%
^{-1}}{\longrightarrow }p_{2}(a_{2}^{\prime }) \leq \cdots \leq p_{n}(a_{n})%
\overset{\mathsf{H}\cup \mathsf{H}^{-1}}{\longrightarrow }%
p_{n}(a_{n}^{\prime }) \leq c^{\prime } \text{,}
\end{equation*}
where $\mathsf{H}^{-1}$ is the inverse relation of $\mathsf{H}$.
\end{lemma}

\begin{remark}
\upshape\label{reflexive}If $\mathsf{H}$ is reflexive then the inequality $%
a\leq b$ gives rise to the scheme $a\leq b\overset{\mathsf{H}}{%
\longrightarrow }b\leq b$, whence $a\Sigma _{H}b$.
\end{remark}

Consider next the ordered algebra $\overline{\mathbf{T}(X)}=(T(X),\,\mathcal{%
F}^{\mathbf{T}(X)}\,,=)$. Clearly $\leq ~=~\leq _{X}\dot{\cup}\leq _{%
\mathcal{F}_{0}}$ is a relation on $T(X)\times T(X)$.

\begin{definition}
\label{Ordered term algebra}\upshape We call the quotient $\widehat{\mathbf{T}%
(X)}= \overline{\mathbf{T}(X)}/\Sigma _{\leq }$ the \textit{ordered term algebra} of
type $\mathcal{F}$ over $(X,\leq_{X}),$ where $\Sigma _{\leq }$
is the compatible quasiorder on $\overline{\mathbf{T}(X)}$ generated by $\leq $.

\end{definition}

In the following theorem we prove that $\Sigma _{\leq }$ is actually an order on $T(X)$. 
Notations introduced in Definition \ref{Term<>} will be instrumental in the proof. 

\begin{definition} \label{Term<>}
\upshape Let $t(z_{1},\ldots ,z_{k})\in T_{\mathcal{F}}^{reg}(V)$ be constant free such that $Tree(t)$ has $k$ leaves and 
$var(t)=(z_{1},\ldots ,z_{k})$. For $a_{1},\ldots ,a_{k}\in V\cup \mathcal{F}%
_{0}$ and $1\leq j\leq k$, we define%
\begin{align*}
t(a_{1},\ldots ,a_{j},\ldots ,a_{k})\left[ j\right] & =a_{j} \\
t(a_{1},\ldots ,a_{j},\ldots ,a_{k})\left\langle l\right\rangle &
=(a_{1},\ldots ,a_{l})\text{; \ \ }l\leq k \\
t(a_{1},\ldots ,a_{j},\ldots ,a_{k})[j,u]& =t(a_{1},\ldots
,a_{j-1},u,a_{j+1},\ldots ,a_{k}).
\end{align*}
\end{definition}

\newpage
\begin{theorem}
$\Sigma _{\leq }\cap \Sigma _{\leq }^{-1}=\triangle _{T(X)}.$
\end{theorem}
\begin{proof}
Let $\left( f,g\right) \in \Sigma _{\leq }\cap \Sigma _{\leq }^{-1}$.
Then, because $\leq $ is reflexive, we have, by Lemma \ref{Induced} and
Remark \ref{reflexive}, the following possibilities.\medskip

\noindent (i) $f=g$, in which case $\left( f,g\right) \in \triangle _{T(X)}$%
.\smallskip

\noindent (ii) There exist schemes%
\begin{equation}
f=p_{1}(y_{1})\overset{\leq }{\longrightarrow }p_{1}(y_{1}^{\prime
})=p_{2}(y_{2})\overset{\leq }{\longrightarrow }p_{2}(y_{2}^{\prime
})=\cdots =p_{n}(y_{n})\overset{\leq }{\longrightarrow }p_{n}(y_{n}^{\prime
})=g,  \label{f<g}
\end{equation}%
\noindent and 
\begin{equation}
g=p_{n+1}(y_{n+1})\overset{\leq }{\longrightarrow }p_{n+1}(y_{n+1}^{\prime
})=\cdots =p_{m}(y_{m})\overset{\leq }{\longrightarrow }p_{m}(y_{m}^{\prime
})=f\text{,}  \label{g<f}
\end{equation}%
 where $p_{i}$ are $X\cup \mathcal{F}%
_{0} $ - cfr translations and $(y_{i},y_{i}^{\prime })\in \,\leq $ for $%
i=1,\dots ,m$. One may also assume by Lemma \ref{RW} that the $X\cup 
\mathcal{F}_{0}$ - cfr translations $p_{i}$, $1\leq i\leq m$, are all induced by
the same constant free regular term, say $t(z_{1},\dots ,z_{k})$.
Using Definition \ref{Term<>}, we can rewrite (\ref{f<g}) and (\ref{g<f}) as
\begin{equation}
\begin{tabular}{lll}
$f=t(a_{1},\dots ,a_{k})$ &  &  \\ 
$\overset{\leq }{\longrightarrow }t(a_{1},\dots ,a_{k})[k_{1},y_{1}^{\prime}]$ & ; & 
$t(a_{1},\dots ,a_{k}) [k_{1}] =y_{1}$ \\ 
$\overset{\leq }{\longrightarrow }t(a_{1},\dots ,a_{k})[k_{1},y_{1}^{\prime
}][k_{2},y_{2}^{\prime }]$ & ; & $t(a_{1},\dots ,a_{k})[k_{1},y_{1}^{\prime
}] [k_{2}] =y_{2}$ \\ 
$\cdots $ & ; &  \\ 
$\overset{\leq }{\longrightarrow }t(a_{1},\dots ,a_{k})[k_{1},y_{1}^{\prime
}]\ldots \lbrack k_{n},y_{n}^{\prime }]=g$ & ; & $t(a_{1},\dots
,a_{k})[k_{1},y_{1}^{\prime }]\ldots \lbrack k_{n-1},y_{n-1}^{\prime
}] [k_{n}] =y_{n}$ \\ 
$\cdots $ &  &  \\ 
$\overset{\leq }{\longrightarrow }t(a_{1},\dots ,a_{k})[k_{1},y_{1}^{\prime
}]\ldots \lbrack k_{m},y_{m}^{\prime }]=f$ & ; & $t(a_{1},\dots
,a_{k})[k_{1},y_{1}^{\prime }]\ldots \lbrack k_{m-1},y_{m-1}^{\prime
}] [k_{m}] =y_{m}\text{.}$%
\end{tabular}
\label{f==g}
\end{equation}
But then, by Lemma \ref{RW}, we must have
\begin{equation}
(a_{1},\dots ,a_{k})=t(a_{1},\dots ,a_{k})[k_{1},y_{1}^{\prime }]\ldots
\lbrack k_{m},y_{m}^{\prime }]\left\langle k\right\rangle
\end{equation}%
This implies that every occurence of $\overset{\leq }{\longrightarrow }$ in (%
\ref{f==g}) may be replaced by $=$, whence we get $f=g$. This also implies that
$\Sigma _{\leq }$ is indeed an order on $T(X)$. 
\end{proof}

We can therefore identify 
$\widehat{\mathbf{T}(X)}=\overline{\mathbf{T}(X)}/\Sigma _{\leq }$ with $(T(X),
\mathcal{F}^{\mathbf{T}(X)},\Sigma _{\leq })$. 
For practical reasons we shall use $\preccurlyeq $ instead of $\Sigma _{\leq }$ to
denote the order of $\widehat{\mathbf{T}(X)}$.
Also, note that for $t_{1},t_{2}\in T(X)$ we have
\begin{equation*}
t_{1}\preccurlyeq t_{2}\quad \mbox{iff}\quad skelt(t_{1})=skelt(t_{2})\,\,
\mbox{and}\,\,leaf(t_{1})[i]\leq leaf(t_{2})[i],
\end{equation*}
where $i=1,\ldots ,n$ and $n=\mbox{the number of leaves of}\,\,t_{1}=
\mbox{the number of leaves of}\,\,t_{2}$.

The following theorem shows that the ordered term algebra $\widehat{\mathbf{T}(X)}$ 
has the universal mapping property for $\mathcal{F}$-\textbf{Oalg}$_{\leq }^{0}$ over $X$.

\begin{theorem}[Cf. Theorem 10.8 of \protect\cite{Burris}]
\label{Universal} Let $\mathcal{F}$ be a type and $(X, \leq_{X} )$ a poset such
that $X\cup \mathcal{F}_{0}\not=\emptyset$. For every ordered algebra $%
\mathbf{D}\in\mathcal{F}$-\textbf{Oalg}$^0_{\leq}$ and every monotone
mapping $\alpha :\left( X,\leq_{X} \right) \longrightarrow (D,\leq _{D})$ there
is a unique homomorphism $\beta :\widehat{\mathbf{T}(X)}\longrightarrow 
\mathbf{D}$, such that the diagram in Figure \ref{universal} commutes.
\begin{figure}[h]
\begin{equation*}
\xymatrix{ X \ar@{_{(}->}[d] \ar[r]^{\alpha} & \mathbf{D} \\
\widehat{\mathbf{T}(X)} \ar[ur]_{\beta} & }
\end{equation*}%
\caption{$\widehat{\mathbf{T}(X)}$ has the universal mapping property for 
$\mathcal{F}$-\textbf{Oalg}$_{\leq }^{0}$ over $X$}
\label{universal}
\end{figure}
\end{theorem}

\begin{proof}
A straightforward adaptation to the ordered context of Theorem 10.8 in \cite{Burris}.
\end{proof}

\section{Amalgamation in $\mathcal{F}$-\textbf{Oalg}$_{\leq }^{0}$} \label{SecAmal}

The idea of the dominion of a subalgebra $\mathbf{B}$ of an algebra $\mathbf{A}$ 
goes back to Isbell \cite{Isbell}. In this section we consider, in the
context of ordered algebras, an order theoretic analogue of the relation between dominions and
 the special amalgamation property.

\begin{definition}
\upshape Let $\mathcal{V}$ be a variety of ordered $\mathcal{F}$-algebras. 
An \textit{amalgam in }$\mathcal{V}$ is a $\mathcal{V}$-diagram given in Figure \ref{4},
\begin{figure}[h]
\begin{equation*}
\xymatrix{
\mathbf{C} \ar[d]_{\phi_2}  \ar[r]^{\phi_1}  & \mathbf{A}    \\
\mathbf{B}                                 &           }
\end{equation*}
\captionof{figure}{Amalgam in $\mathcal{V}$} \label{4}
\end{figure}
such that $\mathbf{A}, \mathbf{B}, \mathbf{C} \in \mathcal{V}$ are pairwise disjoint, 
and $\phi _{1}$, $\phi _{2}\in Hom(\mathcal{V})$, are order-embeddings.
\end{definition}
We denote an amalgam by $(\mathbf{C};\mathbf{A},\mathbf{B}
;\phi _{1},\phi _{2})$ or by an even shorter list $(\mathbf{C};\mathbf{A},
\mathbf{B)}$ if no explicit mention of $\phi _{1}$ and $\phi _{2}$ was
required. If $\mathcal{V}=\mathcal{F}$-\textbf{Oalg}$_{\leq }^{0}$, then we show that
the diagram in Figure \ref{4} can be completed to a pushout (see \cite{MacLane} for definition). 

Because $A\cap
B=\emptyset$, we can consider the poset $(X,\leq _{X})$ where $X=A\dot{\cup}
B$ and $\leq _{X}~=~\leq _{A}\dot{\cup}\leq _{B}$. We first define the
following relations on $T(X)$,
\begin{align*}
\mathsf{R}_{A}& =\left\{ \left( t(x_{1},\ldots ,x_{k}),t^{\mathbf{A}
}(x_{1},\ldots ,x_{k})\right) \,|\,t(x_{1},\ldots ,x_{k})\in T(A)\right\} 
\text{,} \\
\mathsf{R}_{B}& =\left\{ \left( t(y_{1},\ldots ,y_{l}),t^{\mathbf{B}
}(y_{1},\ldots ,y_{l})\right) \,|\,t(y_{1},\ldots ,y_{l})\in T(B)\right\} 
\text{,} \\
\mathsf{H}^{\prime }& =\left\{ (\phi _{1}(c),\phi _{2}(c)\,|\,c\in C\right\} 
\text{,}
\end{align*}%
where $t^{\mathbf{A}}$, $t^{\mathbf{B}}$ are the term functions induced by
the term $t$ on $\mathbf{A}$, $\mathbf{B}$ respectively.  We let
\begin{eqnarray*}
\mathsf{H} &=&\mathsf{R}_{A}\dot{\cup}\mathsf{R}_{B} \\
\widehat{\mathsf{H}} &=&\mathsf{H}\cup \mathsf{H}^{\prime }\text{.}
\end{eqnarray*}

\begin{theorem}
\label{Pushout 1}If $\mathcal{V}=\mathcal{F}$-\textbf{Oalg}$_{\leq }^{0}$,
then the diagram in Figure \ref{4}  can be completed to a pushout.
\end{theorem}
\begin{proof}
Let $\Phi $ be the order-congruence on $\widehat{\mathbf{T}(X)}$ generated
by the relation $\widehat{\mathsf{H}}$. We denote the quotient algebra $
\widehat{\mathbf{T}(X)}/{\Phi }=(T(X)/\Phi ,\mathcal{F}^{\mathbf{T}(X)/\Phi
},\preccurlyeq _{\mathbf{T}(X)/\Phi })$ by $\mathbf{A}\amalg _{\mathbf{C}}
\mathbf{B}$, where we have for $s,t\in T(X)$ 
\begin{equation*}
\lbrack s]_{\Phi }\,\preccurlyeq _{\mathbf{T}(X)/\Phi }\,[t]_{\Phi }\,\,
\mbox{iff}\,\,s\,\underset{\Phi }{\preccurlyeq }\,t.
\end{equation*}%
Let $\chi _{A}:(A, \leq_{A})\longrightarrow (T(X), \preccurlyeq)$ and $\chi _{B}: (B, \leq_{B})\longrightarrow (T(X), \preccurlyeq)$
be the order-embeddings that identify elements of $A$ and $B$ with
their terms in $T(X)$. Also, let $\Phi ^{\natural }:\widehat{\mathbf{T}(X)}%
\longrightarrow \widehat{\mathbf{T}(X)}/{\Phi }$ be the canonical
homomorphism. We define,%
\begin{equation*}
\mu _{1}=\Phi ^{\natural }\circ \chi _{A}:\mathbf{A}\longrightarrow \mathbf{A%
}\amalg _{\mathbf{C}}\mathbf{B}\text{ and }\mu _{2}=\Phi ^{\natural }\circ
\chi _{B}:\mathbf{B}\longrightarrow \mathbf{A}\amalg _{\mathbf{C}}\mathbf{B}%
\text{.}
\end{equation*}%
Now for $f\in \mathcal{F}_{k}$ and $x_{1},\ldots ,x_{k}\in A$, we have:
\begin{eqnarray*}
f^{\widehat{\mathbf{T}(X)}/\Phi }\left( \mu _{1}(x_{1}),\ldots ,\mu
_{1}(x_{k})\right) 
&=&f^{\widehat{\mathbf{T}(X)}/\Phi }\left( \Phi ^{\natural }\circ \chi
_{A}(x_{1}),\ldots ,\Phi ^{\natural }\circ \chi _{A}(x_{k})\right) \medskip
\\
&=&\Phi ^{\natural }\left( f^{\widehat{\mathbf{T}(X)}}\left( \chi
_{A}(x_{1}),\ldots ,\chi _{A}(x_{k})\right) \right) \medskip \\
&=&\Phi ^{\natural }\left( f^{\widehat{\mathbf{T}(X)}}\left( x_{1},\ldots
,x_{k}\right)\right) \medskip \\
&=&\Phi ^{\natural }\left( f(x_{1},\ldots ,x_{k}) \right) \medskip \\
&=&\Phi ^{\natural }\left( f^{\mathbf{A}}(x_{1},\ldots ,x_{k}) \right)\medskip \\
&=&\Phi ^{\natural }\left( \chi _{A}\left( f^{\mathbf{A}}(x_{1},\ldots
,x_{k})\right) \right) \medskip \\
&=&\Phi ^{\natural }\circ \chi _{A}\left( f^{\mathbf{A}}(x_{1},\ldots
,x_{k})\right) \medskip \\
&=&\mu _{1}(f^{\mathbf{A}}(x_{1},\ldots ,x_{k}))\text{.}
\end{eqnarray*}%
This implies that $\mu _{1}$ is a homomorphism of ordered algebras. By a
similar token $\mu _{2}$ also is a homomorphism of ordered algebras.
Furthermore, for any $c\in C$ we have $\mu _{1}\circ \phi _{1}=\mu _{2}\circ
\phi _{2}$, because $[\phi _{1}(c)]_{\Phi }=[\phi _{2}(c)]_{\Phi }$. So, the
object $\mathbf{A}\amalg _{\mathbf{C}}\mathbf{B}$, together with the
morphisms $\mu _{1}$ and $\mu _{2}$, extends the $\mathcal{V}$-diagram in Figure \ref{4}  to the following commutative square.

\begin{figure}[h]
\begin{equation*}
\xymatrix{
\mathbf{C} \ar[d]_{\phi_2}  \ar[r]^{\phi_1}  & \mathbf{A}  \ar[d]_{\mu_1}   \\
\mathbf{B}    \ar[r]^{\mu_2}                   &   \mathbf{A}\amalg _{\mathbf{C}} \mathbf{B}        }
\end{equation*}
\captionof{figure}{} \label{push}
\end{figure}

Next, let $\mathbf{D}$ be an ordered $\mathcal{F}$-algebra admitting
homomorphisms $\gamma _{A}:\mathbf{A}\longrightarrow \mathbf{D}$ and $\gamma
_{B}:\mathbf{B}\longrightarrow \mathbf{D}$, such that $\gamma _{A}\circ \phi
_{1}=\gamma _{B}\circ \phi _{2}$. Then considering the monotone map%
\begin{equation*}
\alpha =\gamma _{A}\dot{\cup}\gamma _{B}:(X,\leq _{X})\longrightarrow
(D,\leq _{D})\text{,}
\end{equation*}%
there exists, by Theorem \ref{Universal}, a unique homomorphism $\beta :%
\widehat{\mathbf{T}(X)}\longrightarrow \mathbf{D}$ such that the diagram in
Figure \ref{universal} commutes. This also implies that $\beta \left\vert _{A}\right.
=\gamma _{A}$, $\beta \left\vert _{B}\right. =\gamma _{B}$. Let us first
prove that $\underset{\Phi }{\preccurlyeq }\,\subseteq \,\overset{%
\longrightarrow }{\ker }\beta $.

Suppose $s\underset{\Phi }{\preccurlyeq }t$, for $s,t\in T(X)$. Then by Lemma 
\ref{Generated} either $s\preccurlyeq t$, in which case $\beta (s)\leq
_{D}\beta (t)$ and we are done, or there exists a scheme
\begin{equation}
s\preccurlyeq p_{1}(x_{1})\overset{\widehat{\mathsf{H}}\dot{\cup}\widehat{\mathsf{H}}^{-1}}{\longrightarrow }p_{1}(y_{1})
\preccurlyeq p_{2}(x_{2})
\overset{\widehat{\mathsf{H}}\dot{\cup}\widehat{\mathsf{H}}^{-1}}{%
\longrightarrow }p_{2}(y_{2})\preccurlyeq \cdots \preccurlyeq p_{n}(x_{n})%
\overset{\widehat{\mathsf{H}}\dot{\cup}\widehat{\mathsf{H}}^{-1}}{%
\longrightarrow }p_{n}(y_{n})\preccurlyeq t\text{.}  \label{Scheme}
\end{equation}%
Using the monotonicity of $\beta $, the inequalities $s\preccurlyeq
p_{1}(x_{1})$, $p_{i}(y_{i})\preccurlyeq p_{i+1}(x_{i+1})$, where $1\leq
i\leq n-1$, and $p_{n}(y_{n})\preccurlyeq t$ imply that $\beta (s)\leq
_{D}\beta (p_{1}(x_{1}))$, $\beta (p_{i}(y_{i}))\leq _{D}\beta
(p_{i+1}(x_{i+1}))$ and $\beta (p_{n}(y_{n}))\leq _{D}\beta (t)$,
respectively. Also, observe that every pair $(x_{i},y_{i})\in \mathsf{H}\dot{%
\cup}\mathsf{H}^{-1}$ comprises a term and its value, either in $\mathbf{A}$
or in $\mathbf{B}$. This implies that $\beta (x_{i})=\beta (y_{i})$. On the other hand,
if $(x_{i},y_{i})\in \mathsf{H}^{\prime }$, then $x_{i}=\phi _{1}(c)$, $%
y_{i}=\phi _{2}(c)$ for some $c\in C$ and we have%
\begin{equation*}
\beta (x_{i})=\beta (\phi _{1}(c))=\gamma _{A}(\phi _{1}(c))=\gamma
_{B}(\phi _{2}(c))=\beta (\phi _{2}(c))=\beta (y_{i}).
\end{equation*}%
Similarly, $(x_{i},y_{i})\in {\mathsf{H}^{\prime }}^{-1}$ implies $\beta
(x_{i})=\beta (y_{i})$. So, for any $(x_{i},y_{i})\in \widehat{\mathsf{H}}\,%
\dot{\cup}\,\widehat{\mathsf{H}}^{-1}$ we have $\beta (x_{i})=\beta (y_{i})$%
. The relation $p_{i}(x_{i})\overset{\widehat{\mathsf{H}}\,\dot{\cup}\,%
\widehat{\mathsf{H}}^{-1}}{\longrightarrow }p_{i}(y_{i})$, $1\leq i\leq n$,
therefore implies that $\beta (p_{i}(x_{i}))=\beta (p_{i}(y_{i}))$. Summing
up the discussion, we can write $\beta (s)\leq _{D}\beta (t)$ from scheme %
(\ref{Scheme}). Hence $\underset{\Phi }{\preccurlyeq }\,\subseteq \,\overset{%
\longrightarrow }{\ker }\beta $.

Now, by Theorem \ref{Homomorphism}, there exists a unique homomorphism $%
\delta $ such that the diagram in Figure \ref{2*} commutes. 

\begin{figure}[h]
\begin{equation*}
\xymatrix{
\widehat{\mathbf{T}(X)} \ar[d]_{\Phi ^{\natural }}  \ar[r]^{\beta}  & \mathbf{D}    \\
\mathbf{A}\amalg_{\mathbf{C}} \mathbf{B}        \ar[ur]_{\delta}                            &           }
\end{equation*}
\captionof{figure}{} \label{2*}
\end{figure}

\noindent Let us next consider the diagram in Figure \ref{3}.

\begin{figure}[h]
\begin{equation*}
\xymatrix{ \mathbf{C} \ar[rr]^{\phi_1} \ar[dd]_{\phi_2}& & \mathbf{A} \ar[dd]_{\mu_1} \ar[dl]_{\chi _{A}}
\ar[dddr]^{\gamma_{A}} & \\ & \widehat{\mathbf{T}(X)} \ar[dr]^{\Phi
^{\natural }} & & \\ \mathbf{B} \ar[ur]^{\chi _{B}} \ar[rr]^{\mu_2}
\ar[drrr]_{\gamma_{B}} & & \mathbf{A}\amalg_ {\mathbf{C}} \mathbf{B} \ar[dr]^{\delta} & \\
& & & \mathbf{D} } \text{.}
\end{equation*}%
\caption{}
\label{3}
\end{figure}
\noindent For $x\in A$, we have 
\begin{eqnarray*}
\delta \circ \mu _{1}(x) &=&\delta \circ (\Phi ^{\natural }\circ \chi
_{A})(x) \\
&=&((\delta \circ \Phi ^{\natural })\circ \chi _{A})(x) \\
&=&\beta \circ \chi _{A}(x) \\
&=&\gamma _{A}(x)\text{.}
\end{eqnarray*}%
Similarly for $y\in B$ one can show that 
\begin{equation*}
\delta \circ \mu _{2}(x)=\gamma _{B}(x).
\end{equation*}%
This implies that the diagram in Figure \ref{3} commutes and the proof is
completed.
\end{proof}

\noindent An amalgam $(\mathbf{C};\mathbf{A},\mathbf{B};\phi _{1},\phi _{2})$ is said to be \textit{embeddable} if the following conditions are satisfied.
\begin{enumerate}
\item $\mu _{1}$ and $\mu _{2}$ in Figure \ref{push} are order-embeddings. 
\item $\mu_{1}(x)=\mu_{2}(y)$, for $x\in A$, $y\in B$, implies $x=\phi_{1}(c)$, $y=\phi _{2}(c)$, for some $c\in C$.
\end{enumerate}

\noindent If only condition (1) is satisfied, then we say that $(\mathbf{C};\mathbf{A},\mathbf{B};\phi _{1},\phi _{2})$
is \textit{weakly embeddable}. 

\begin{definition}
\upshape We call  $(\mathbf{C};\mathbf{A}_{1},\mathbf{A}_{2};\phi _{1},\phi _{2})$ a \textit{special amalgam} if $\mathbf{A}_{1}$ is isomorphic to 
$\mathbf{A}_{2}$ via, say, $\nu :\mathbf{A}_{1}\longrightarrow \mathbf{A}_{2}$, such that $\phi _{2}=\nu \circ \phi _{1}$.
\end{definition}

\noindent One can easily verify that every special amalgam $(\mathbf{C};\mathbf{A}_{1},\mathbf{A}_{2})$ is weakly embeddable.

We shall say that a variety $\mathcal{V}$ of ordered algebras (with pushouts) has the \textit{(weak) amalgamation property}
if every amalgam in $\mathcal{V}$ is (weakly) embeddable. We say that $\mathcal{V}$ has the \textit{special amalgamation property} if every special
amalgam in $\mathcal{V}$ is embeddable.

\begin{definition}
\upshape Let $\mathbf{B}$ be a subalgebra of an ordered $\mathcal{F}$-algebra 
$\mathbf{A}$. The \textit{dominion} of $\mathbf{B}$ in $\mathbf{A}$, denoted by $%
\widehat{Dom}_{\mathbf{A}}\mathbf{B}$, is the set of all elements $d\in A$
such that for every pair of homomorphisms of ordered $\mathcal{F}$-algebras $%
f,g:\mathbf{A} \longrightarrow \mathbf{C}$ satifying $f\left\vert
_{B}\right. =g\left\vert _{B}\right.$ we have $f(d)=g(d)$.
\end{definition}

\noindent Clearly, $\widehat{Dom}_{\mathbf{A}}\mathbf{B}$ contains $B$ and it is
actually a subalgebra of $\mathbf{A}$. We say that $\mathbf{B}$ is 
\textit{closed} in $\mathbf{A}$ if $\widehat{Dom}_{\mathbf{A}}\mathbf{B}=\mathbf{B}$.
We call $\mathbf{B}\in \mathcal{V}$  \textit{absolutely closed} if $\widehat{Dom}_{\mathbf{A}}\mathbf{B}=\mathbf{B}$ for every 
order-embedding 
$\mathbf{B}\hookrightarrow \mathbf{A},$ where $\mathbf{A}\in \mathcal{V}.$
A variety $\mathcal{V}$ is said to be \textit{closed} if every $\mathbf{A}\in \mathcal{V}$ is absolutely closed. Dominions are related to epimorphisms: $f:%
\mathbf{A}\longrightarrow \mathbf{B}$ is an epimorphism iff $\widehat{Dom}_{%
\mathbf{B}}\,\mathbf{Im}(f)=\mathbf{B}$. By disregarding the ordering of
algebras one can also consider the \textit{algebraic dominion} of $\mathbf{B}
$ in $\mathbf{A}$, which we shall denote by $Dom_{\mathbf{A}}\mathbf{B}$.
One can easily verify that%
\begin{equation*}
B\subseteq Dom_{\mathbf{A}}\mathbf{B}\subseteq \widehat{Dom}_{\mathbf{A}}%
\mathbf{B}\subseteq A.
\end{equation*}%
It was shown in \cite{NS} that in the categories of all pomonoids and all
posemigroups $\widehat{Dom}_{\mathbf{A}}\mathbf{B}\subseteq Dom_{\mathbf{A}}%
\mathbf{B}$. In our main result, Theorem \ref{main}, we generalize
this fact to the categories $\mathcal{F}$-\textbf{Oalg}$_{\leq }^{0}$.

If $\mathbf{C}$ is a subalgebra of an ordered algebra $\mathbf{A}$, and if $\mathbf{A}_{1}, \mathbf{A}_{2}$ are disjoint isomorphic copies 
of $\mathbf{A}$ via isomorphisms $\alpha _{i}:\mathbf{A}\longrightarrow 
\mathbf{A}_{i}$ for $i=1,2$, then $\left( \mathbf{C};\mathbf{A}_{1},\mathbf{A%
}_{2};\phi_{1},\phi _{2}\right)$ is a special amalgam, where $\phi
_{i}=\alpha _{i}\left\vert _{C}\right.$ and $\nu =\alpha _{2}\circ \alpha _{1}^{-1}:\mathbf{A}%
_{1}\longrightarrow \mathbf{A}_{2}$. So, we have a
commutative diagram given in Figure \ref{spalg}.

\begin{figure}[h]
\begin{equation*}
\xymatrix{ 
 \mathbf{A} \ar[drr]^{\alpha_1} \ar[ddr]_{\alpha_2}&     & \\
& \mathbf{C} \ar@{_{(}->}[ul] \ar[d]^{\phi_2} \ar[r]^{\phi_1} & \mathbf{A}_1 \ar[d]_{\mu_1} \ar[dl]_{\nu} \\ 
& \mathbf{A}_2 \ar[r]^{\mu_2} & \mathbf{A}_1\amalg_{\mathbf{C}} \mathbf{A}_2 
}
\end{equation*}
\captionof{figure}{} \label{spalg}
\end{figure}

\noindent Let $X=A_1\cup A_2$, and consider the monotone map $\nu_{0}:X\longrightarrow A_{1}$ given
by 
\begin{equation*}
\nu _{0}(x)=\left\{ 
\begin{tabular}{ll}
$x$ & if $x\in A_{1}$ \\ 
$\nu ^{-1}(x)$ & if $x\in A_{2}$%
\end{tabular}
\right. \text{.}
\end{equation*}
By Theorem \ref{Universal}, there exists a unique homomorphism $\overline{\nu%
}_{0} :\widehat{\mathbf{T}(X)}\longrightarrow {\mathbf{A}_{1}}$ such that $%
\overline{\nu}_{0}\left\vert _{X}\right. =\nu _{0}$.

\begin{proposition}
\label{Prop1} Let $\left(\mathbf{C};\mathbf{A}_{1},\mathbf{A}_{2};\phi _{1},\phi _{2}\right)$
be a special amalgam as defined by Figure \ref{spalg}. If $\mu _{1}\left(
x_{1}\right) =\mu _{2}\left( x_{2}\right) $ for some $x_{1}\in A_{1}$, $%
x_{2}\in A_{2}$, then $x_{1}=\alpha _{1}(x)$, and $x_{2}=\alpha _{2}(x)$ for some $x\in A$.
\end{proposition}
\begin{proof}
Let $\mu _{1}\left( x_{1}\right) =\mu _{2}\left( x_{2}\right) $. Then $[x_1]_{\Phi}=[x_2]_{\Phi}$ (refer to the proof of Theorem \ref{Pushout 1}),
and by Lemma \ref{Generated} we have the schemes 
\begin{equation}
x_{1}\preccurlyeq p_{1}(y_{1})\overset{\widehat{\mathsf{H}}\cup \widehat{%
\mathsf{H}}^{-1}}{\longrightarrow }p_{1}(y_{1}^{\prime })\preccurlyeq
p_{2}(y_{2})\overset{\widehat{\mathsf{H}}\cup \widehat{\mathsf{H}}^{-1}}{%
\longrightarrow }p_{2}(y_{2}^{\prime })\preccurlyeq \cdots \preccurlyeq p_{n}(y_{n})%
\overset{\widehat{\mathsf{H}}\cup \widehat{\mathsf{H}}^{-1}}{\longrightarrow 
} p_{n}(y_{n}^{\prime })\preccurlyeq x_{2}\text{,}  
\label{Domsch1}
\end{equation}
\begin{equation}
x_{2}\preccurlyeq p_{n+1}(y_{n+1}^{\prime })\overset{\widehat{\mathsf{H}}%
\cup \widehat{\mathsf{H}}^{-1}}{\longrightarrow }p_{n+1}(y_{n+1})%
\preccurlyeq \cdots \preccurlyeq p_{n+k}(y_{n+k}^{\prime })\overset{\widehat{%
\mathsf{H}}\cup \widehat{\mathsf{H}}^{-1}}{\longrightarrow }%
p_{n+k}(y_{n+k})\preccurlyeq x_{1}.
\label{Domsch2}
\end{equation}

Let us consider scheme (\ref{Domsch1}). 
By the monotonicity of $\overline{\nu}_{0}$, the inequalities $%
x_1\preccurlyeq p_{1}(y_{1})$, $p_{i}(y_{i}^{\prime})\preccurlyeq
p_{i+1}(y_{i+1})$, for $1\leq i\leq n-1$, and $p_{n}(y_{n}^{\prime})%
\preccurlyeq x_2$ imply $\overline{\nu}_{0}(x_1)\leq_{A_1} \overline{\nu}%
_{0}(p_{1}(y_{1}))$, $\overline{\nu}_{0}(p_{i}(y_{i}^{\prime}))\leq_{A_1} 
\overline{\nu}_{0}(p_{i+1}(y_{i+1}))$, and $\overline{\nu}%
_{0}(p_{n}(y_{n}^{\prime}))\leq_{A_1} \overline{\nu}_{0}(x_2)$ respectively.
For an arbitrary $(y_i, y_i^{\prime})\in {\widehat{\mathsf{H}}\cup \widehat{\mathsf{H}}^{-1}}$ we consider two cases. 

{\bf Case 1:} $(y_i, y_i^{\prime})\in \mathsf{H}\cup \mathsf{H}^{-1}$. Suppose $(y_i,
y_i^{\prime})\in \mathsf{H}$, then $(y_i, y_i^{\prime})=(t(z_1,\dots, z_l),
t^{\mathbf{A}_{j}}(z_1,\dots, z_l))$, where $z_1,\dots, z_l\in A_{j}$ and $j\in \{1,2\}$. 
If $z_1,\dots, z_l\in A_{1}$, then 
$\overline{\nu}_{0}(y_i)=t^{\mathbf{A}_{1}}(z_1,\dots, z_l)=\overline{\nu}_{0}(y_i^{\prime}).$ 
If $z_1,\dots, z_l\in A_{2}$, then 
\[ \overline{\nu}_{0}(y_i)= \overline{\nu}_{0}(t(z_1,\dots, z_l))=t^{\mathbf{A}_{1}}(\overline{\nu}_{0} (z_1),\dots, \overline{\nu}_{0}(z_l))
=t^{\mathbf{A}_{1}}(\nu^{-1}(z_1),\dots, \nu^{-1}(z_l)), \]
and 
$\overline{\nu}_{0}(y_i^{\prime})=\nu^{-1}(t^{\mathbf{A}_{2}}(z_1,\dots, z_l))$. Since 
$\nu^{-1}$ is an isomorphism we have 
\[\nu^{-1}(t^{\mathbf{A}_{2}}(z_1,\dots, z_l))=t^{\mathbf{A}_{1}}(\nu^{-1}(z_1),\dots, \nu^{-1}(z_l))\] 
and therefore $\overline{\nu}_{0}(y_i)=\overline{\nu}_{0}(y_i^{\prime})$. If $(y_i, y_i^{\prime})\in \mathsf{H}^{-1}$, then
$(y_i^{\prime}, y_i)\in \mathsf{H}$ and the above argument applies. 

{\bf Case 2:} $(y_i, y_i^{\prime})\in \mathsf{H}^{\prime}\cup {\mathsf{H}^{\prime}}%
^{-1}$. If $(y_i, y_i^{\prime})\in \mathsf{H}^{\prime}$, then $y_i=\phi_1(c)$%
, $y_i^{\prime}=\phi_2(c)$ for some $c\in C$, and we have $\overline{\nu}%
_{0}(y_i)=\overline{\nu}_{0}(\phi_1(c))=\phi_1(c)$, $\overline{\nu}%
_{0}(y_i^{\prime})=\overline{\nu}_{0}(\phi_2(c))=\nu^{-1}\phi_2(c)=\phi_1(c)$%
, i.e., $\overline{\nu}_{0}(y_i)=\overline{\nu}_{0}(y_i^{\prime})$.
Similarly, $(y_i, y_i^{\prime})\in {\mathsf{H}^{\prime}}^{-1}$ implies $%
\overline{\nu}_{0}(y_i)=\overline{\nu}_{0}(y_i^{\prime})$.

So, we conclude that the relations $p_{i}(y_{i})\overset{ \widehat{\mathsf{H}}\,\cup\,%
\widehat{\mathsf{H}}^{-1}}{\longrightarrow }p_{i}(y_{i}^{\prime})$ imply $%
\overline{\nu}_{0}(p_{i}(y_{i}))=\overline{\nu}_{0}(p_{i}(y_{i}^{\prime}))$.
One may therefore write the following sequence from scheme (\ref{Domsch1}):
\begin{equation}
x_{1}\leq_{A_1} \overline{\nu}_{0}(p_{1}(y_{1}))= \overline{\nu}_{0}(p_{1}(y_{1}^{\prime }))\leq_{A_1}
\cdots \leq_{A_1} \overline{\nu}_{0}(p_{n}(y_{n}))
= \overline{\nu}_{0}(p_{n}(y_{n}^{\prime }))\leq_{A_1} \overline{\nu}_{0}(x_{2})\text{.}  
\label{Domsch1'}
\end{equation}

\noindent Using the same reasoning, one gets 
\begin{equation}
\overline{\nu}_{0} (x_2)\leq _{A_1}\overline{\nu}_{0}(x_1)=x_1, 
\label{Domsch2'}
\end{equation}
from scheme (\ref{Domsch2}).
Combining (\ref{Domsch1'}) and (\ref{Domsch2'}) we have $x_1=\overline{\nu}_{0} (x_2)$.
Finally, taking $x_{1}=\alpha _{1}(x)$, $x\in A$, we get 
\[ x_{2} =\nu (x_{1}) =\nu \alpha _{1}(x) =\alpha _{2}(x),\]
as required. 
\end{proof}

\begin{proposition} \label{PropDom}
Let $\mathbf{C}$ be a subalgebra of an ordered algebra $\mathbf{A}$. Let $\alpha _{i}:\mathbf{A}\longrightarrow \mathbf{A}_{i}$ for $i=1,2$ be
isomorphisms and let $\alpha _{i}\left\vert _{C}\right. =\phi _{i}: \mathbf{C}\longrightarrow \mathbf{A}_{i}$. Then considering the special
amalgam $(\mathbf{C};\mathbf{A}_{1},\mathbf{A}_{2};\phi _{1},\phi_{2})$ we
have 
\begin{equation*}
\widehat{\text{Dom}}_{\mathbf{A}} \mathbf{C}  \cong \widehat{ 
\text{Dom}}_{\mathbf{A}_{i}}  \phi _{i}\left( \mathbf{C}\right) 
=\mu_{i}^{-1}\left[ \mu _{1}\left( \mathbf{A}_{1}\right) \cap \mu _{2}\left( 
\mathbf{A}_{2}\right) \right],
\end{equation*}
were $\mu _{i}$ for $i=1,2$ are as defined in the proof of Theorem \ref{Pushout 1}. 
\end{proposition}
\begin{proof}
Obviously $\widehat{\text{Dom}}_{\mathbf{A}} \mathbf{C} \cong 
\widehat{\text{Dom}}_{\mathbf{A}_{i}} \phi _{i}\left( \mathbf{C}\right)$. Also, it suffices to prove that
\begin{equation*}
\widehat{\text{Dom}}_{\mathbf{A}_{1}} \phi _{1}\left( \mathbf{C}\right)  =\mu _{1}^{-1}\left[ \mu _{1}\left( \mathbf{A}_{1}\right)
\cap \mu _{2}\left( \mathbf{A}_{2}\right) \right] .
\end{equation*}%
Consider $\nu :\mathbf{A}_{1}\longrightarrow \mathbf{A}_{2}$, as defined in Figure \ref{spalg}.
 Let $x\in \widehat{\text{Dom}}_{\mathbf{A}_{1}} \phi _{1}\left( \mathbf{C}\right)$ and let 
$y\in \mathbf{C}$. Observe that
\begin{equation*}
\begin{tabular}{lll}
$\mu _{1}\left( \phi _{1}(y)\right) $ & $=\mu _{1}\circ \phi _{1}(y)$ &  \\ 
& $=\mu _{2}\circ \phi _{2}(y)$ & $\text{by commutativity of the diagram (%
\ref{spalg})}$ \\ 
& $=\mu _{2}\left( \nu \circ \phi _{1}(y)\right) $ & $\text{because }\phi
_{2}=\nu \circ \phi _{1}$ \\ 
& $=\mu _{2}\circ \nu \left( \phi _{1}\left( y\right) \right) $.& 
\end{tabular}%
\end{equation*}%
So, $\mu _{1}$ and $\mu _{2}\circ \nu $ agree on $\phi _{1}\left( \mathbf{C}%
\right) $. This implies that
$\mu _{1}\left( x\right) =\mu _{2}(\nu (x)),$
whence $\mu _{1}(x)\in \mu _{1}\left( \mathbf{A}_{1}\right) \cap \mu
_{2}\left( \mathbf{A}_{2}\right) $, i.e., 
\begin{equation*}
x\in \mu _{1}^{-1}\left[ \mu _{1}\left( \mathbf{A}_{1}\right) \cap \mu
_{2}\left( \mathbf{A}_{2}\right) \right] .
\end{equation*}%
Now, suppose that $x_{1}\in \mu _{1}^{-1}\left[ \mu _{1}\left( 
\mathbf{A}_{1}\right) \cap \mu _{2}\left( \mathbf{A}_{2}\right) \right] $.
This implies that 
\begin{equation*}
\mu _{1}(x_{1})\in \mu _{1}\left( \mathbf{A}_{1}\right) \cap \mu _{2}\left( 
\mathbf{A}_{2}\right) .
\end{equation*}%
Let $\mu _{1}(x_{1})=\mu _{2}(x_{2})$, $x_{2}\in A_{2}$. Then by Proposition \ref{Prop1} we have $x\in A$ such that $x_{1}=\alpha _{1}(x)$, $x_{2}=\alpha
_{2}(x)$. This yields, 
\begin{align*}
\mu _{1}\left( x_{1}\right) & =\mu _{1}(\alpha _{1}(x)) \\
& =\mu _{2}(\alpha _{2}(x)) \\
& =\mu _{2}\circ \alpha _{2}(x) \\
& =\mu _{2}(\nu \circ \alpha _{1}(x)) \\
& =\mu _{2}\circ \nu (x_{1}).
\end{align*}
So, we have 
\begin{equation}\label{mu1=mu2nu}
\mu _{1}\left( x_{1}\right) =\mu _{2}\circ \nu (x_{1}).
\end{equation}

Next, let $\mathbf{B}$ be an arbitrary ordered algebra admitting homomorphisms 
$f,g:\mathbf{A}_{1}\longrightarrow \mathbf{B}$ such that $f \circ \phi_1=g\circ \phi_1$ 
Define $g^{\prime }:\mathbf{A}_{2}\longrightarrow \mathbf{B}$
by $g^{\prime }=g\circ \nu ^{-1}$, and consider the commutative diagram in Figure \ref{Dom_}, where $\psi$ is the unique homomorphism given by the pushout 
$\mathbf{A}_1\amalg _{\mathbf{C}} \mathbf{A}_2$.
\begin{figure}[h]
\begin{equation*}
\xymatrix{ \mathbf{C} \ar[d]_{\phi_2} \ar[r]^{\phi_1} & \mathbf{A}_1
\ar[d]_{\mu_1} \ar[dl]_{\nu} \ar[ddr]^{f} & \\ \mathbf{A}_2 \ar[r]^{\mu_2}
\ar[drr]_{g^{\prime}} & \mathbf{A}_1\amalg _{\mathbf{C}} \mathbf{A}_2
\ar[dr]^{\psi} & \\ & & \mathbf{B} }
\end{equation*}
\captionof{figure}{} \label{Dom_}
\end{figure}

\noindent Now calculate 
\begin{eqnarray*}
f(x_{1}) &=&\psi\circ \mu _{1}\left( x_{1}\right) \\
&=&\psi\circ( \mu _{2}\circ \nu \left( x_{1}\right)) \text{ \ \ \ \ \ \ using (\ref{mu1=mu2nu})} \\
&=&g^{\prime }\circ \nu \left( x_{1}\right) \\
&=&g(x_{1})\text{,}
\end{eqnarray*}%
whence $x_{1}\in \widehat{\text{Dom}}_{\mathbf{A}_{1}} \phi _{1}(\mathbf{C}),$ as required.
\end{proof}

\begin{corollary} \label{cor4Eq}
\label{Sp.amalgamation} The following conditions are equivalent:

\begin{itemize}
\item[1)] a special amalgam $(\mathbf{C}; \mathbf{A}_{1},\mathbf{A}_{2};\phi
_{1},\phi _{2})$ is embeddable,

\item[2)] $\mu _{i}^{-1}\left[ \mu _{1}\left( \mathbf{A}_{1}\right) \cap
\mu _{2}\left( \mathbf{A}_{2}\right) \right] =\phi _{i}(\mathbf{C})$, where $\mu _{i}$, $\phi _{i}$, $1\leq i \leq 2$, are as given in 
Figure \ref{spalg}. 

\item[3)] $\widehat{\text{Dom}}_{\mathbf{A}} \mathbf{C} \cong 
\widehat{\text{Dom}}_{\mathbf{A}_{i}}  \phi _{i}\left( \mathbf{C} \right)=\phi _{i}(\mathbf{C})\cong \mathbf{C}$, where both of the
isomorphisms are the restrictions of $\alpha _{i}:\mathbf{A}\longrightarrow \mathbf{A}_{i}$ as defined in Figure \ref{spalg}, 

\item[4)] $\mathbf{C}$ is closed in $\mathbf{A}$.
\end{itemize}
\end{corollary}

\begin{proof}
\noindent (1) $\Longleftrightarrow $ (2) is obvious.

\noindent (2) $\Longleftrightarrow $ (3) follows from Proposition \ref{PropDom}.

\noindent (3) $\Longrightarrow $ (4) Let $\widehat{\text{Dom}}_{\mathbf{A}}
\mathbf{C} \cong \mathbf{C}$. 
Suppose that there exists
$x\in \widehat{\text{Dom}}_{\mathbf{A}} \mathbf{C}\setminus \mathbf{C}$. Then, because $\alpha_i (x)\not\in \alpha _{i}(\mathbf{C})=\phi _{i}(\mathbf{C})$, we have by (3) 
$\alpha_i (x)\not\in  \widehat{\text{Dom}}_{\mathbf{A}_i} \phi _{i}(\mathbf{C})$. This implies that 
$x\not\in \alpha_{i}^{-1} [ \widehat{\text{Dom}}_{\mathbf{A}_i} \phi _{i}(\mathbf{C})]= \widehat{\text{Dom}}_{\mathbf{A}} \mathbf{C},$ contradiction.
Thus, $\widehat{\text{Dom}}_{\mathbf{A}} \mathbf{C} =\mathbf{C}$.

\noindent (4) $\Longrightarrow $ (3) If $\mathbf{C}$ is closed in $\mathbf{A}$, then $%
\widehat{\text{Dom}}_{\mathbf{A}} \mathbf{C} =\mathbf{C}\cong
\phi _{i}(\mathbf{C})$. Since $\phi _{i}(\mathbf{C})\subseteq \widehat{\text{Dom}}_{\mathbf{A}_{i}} \phi _{i}\left( \mathbf{C} \right)$,
the proof will be accomplished if we show that
\begin{equation*}
\widehat{\text{Dom}}_{\mathbf{A}_{i}} \phi _{i}\left( \mathbf{C}%
\right) \subseteq \phi _{i}(\mathbf{C})\,\,\text{for } i=1, 2\text{.}
\end{equation*}
To this end, let $\alpha _{i}(x)\in \widehat{\text{Dom}}_{\mathbf{A}_{i}} \phi _{i}\left( \mathbf{C}\right),$ where $x\in \mathbf{A}$. Now,
if $f,g:\mathbf{A}\longrightarrow \mathbf{B}$ are homomorphisms of ordered
algebras, with $f\left\vert _{\mathbf{C}}\right. =g\left\vert _{\mathbf{C}%
}\right. $, then for
\begin{equation*}
f\circ \alpha _{i}^{-1},g\circ \alpha _{i}^{-1}:\mathbf{A}%
_{i}\longrightarrow \mathbf{B}, \quad \text{(as defined in Figures (\ref{spalg}) and (\ref{Dom_}))}
\end{equation*}%
we have $f\circ \phi _{i}^{-1}(\phi _{i}\left( c\right) )=g\circ \phi
_{i}^{-1}(\phi _{i}\left( c\right) ),$ where $c\in \mathbf{C}$. Because $\alpha
_{i}(x)\in \widehat{\text{Dom}}_{\mathbf{A}_{i}} \phi _{i}\left( 
\mathbf{C}\right),$ we have $f\circ \alpha
_{i}^{-1}(\alpha _{i}(x))=g\circ \alpha _{i}^{-1}(\alpha _{i}(x))$, whence $%
f(x)=g(x)$. This implies that $x\in \widehat{\text{Dom}}_{\mathbf{A}} \mathbf{C}=\mathbf{C}$ and so 
$\alpha _{i}(x)\in \alpha _{i}(\mathbf{C})=\phi _{i}(\mathbf{C})$, as required.
\end{proof}

\noindent The following two propositions relate special amalgamation, epis and absolute closure in varieties of ordered algebras. 

\begin{proposition}
Let $\mathcal{V}$ be a variety of ordered algebras. Then epis are surjective
in $\mathcal{V}$ iff $\mathcal{V}$ is closed.
\end{proposition}
\begin{proof}
($\Longrightarrow $) Let epis be surjective in $\mathcal{V}$. Let $\mathbf{C}$ be a subalgebra of an ordered algebra $\mathbf{A}\in \mathcal{V}$.
Consider the embedding $\chi  :\mathbf{C}\hookrightarrow \widehat{\text{Dom}}_{\mathbf{A}} \mathbf{C}$, that is clealry an epi and hence
surjective. But then $\widehat{\text{Dom}}_{\mathbf{A}}\mathbf{C}=\mathbf{C}$, as required.

($\Longleftarrow $) Let $\mathcal{V}$ be closed. If $f:\mathbf{A}%
\longrightarrow \mathbf{B}$ is an epi in $\mathcal{V}$, then $\widehat{%
\text{Dom}}_{\mathbf{B}} \,Im f =\mathbf{B}$. But,
because $\mathcal{V}$ is closed, we also have $\widehat{\text{Dom}}_{\mathbf{B}}\, Im f =\, Im f$. So, $\,Im f=\mathbf{B}$.
\end{proof}

\begin{proposition}
Let $\mathcal{V}$ be a variety of ordered algebras. Then $\mathcal{V}$ is
closed iff $\mathcal{V}$ has the special amalgamation property.
\end{proposition}
\begin{proof}
($\Longrightarrow $) Let $\mathcal{V}$ be closed, and let $(\mathbf{C};%
\mathbf{A}_{1},\mathbf{A}_{2})$ be a special amalgam in $\mathcal{V}$.
Then 
\begin{alignat}{2}
\mu _{i}^{-1}\left[ \mu _{1}\left( \mathbf{A}_{1}\right) \cap \mu
_{2}\left( \mathbf{A}_{2}\right) \right] &= \widehat{\text{Dom}}_{\mathbf{A
}_{i}} \phi _{i}\left( \mathbf{C}\right),& &\quad  \text{by Proposition \ref{PropDom}} \\
&=\phi _{i}\left( \mathbf{C}\right),&  & \quad \text{because }\mathcal{V}\text{
is closed.}
\end{alignat}
This implies, by Corollary \ref{cor4Eq}, that $(\mathbf{C};\mathbf{A}_{1},\mathbf{A}_{2})$ is
embeddable.

($\Longleftarrow $) On the other hand, suppose $\mathcal{V}$ has the special
amalgamation property. Let $\mathbf{C}$ be a subalgebra of an ordered
algebra $\mathbf{A}$ in $\mathcal{V}$, giving rise to a special amalgam $(%
\mathbf{C};\mathbf{A}_{1},\mathbf{A}_{2})$ in $\mathcal{V}$ as described in Figure \ref{spalg}. Then, first observe that,
\begin{alignat}{2}
\widehat{\text{Dom}}_{\mathbf{A}_{i}} \phi _{i}\left( \mathbf{C}
\right)  &=\mu _{i}^{-1}\left[ \mu _{1}\left( \mathbf{A}_{1}\right)
\cap \mu _{2}\left( \mathbf{A}_{2}\right) \right],  & &\quad \text{by Proposition \ref{PropDom}} \\
&=\phi _{i}\left( \mathbf{C}\right), & &\quad \text{by Corollary \ref{cor4Eq}. }
\end{alignat}

Now, suppose $x\in \widehat{\text{Dom}}_{\mathbf{A}} \mathbf{C} $%
. Let $f,g:\mathbf{A}_{i}\longrightarrow \mathbf{B}$ be such that $%
f\left\vert _{\phi _{i}\left( \mathbf{C}\right) }\right. =g\left\vert
_{\phi _{i}\left( \mathbf{C}\right) }\right. $, i.e. $f\circ \phi
_{i}\left( y\right) =g\circ \phi _{i}\left( y\right) $, for all $y\in 
\mathbf{C}$. This implies that 
$f\circ \alpha _{i}(y)=f\circ \phi_{i}\left( y\right) =g\circ \phi _{i}\left( y\right) =g\circ \alpha _{i}(y).$
Because $x\in \widehat{\text{Dom}}_{\mathbf{A}}\mathbf{C}
 \mathbf{\ }$we have 
 $f\left( \alpha _{i}\left( x\right) \right)=g\left( \alpha _{i}\left(x\right) \right) $. 
 Thus $x\in \widehat{\text{Dom}}_{\mathbf{A}}
\mathbf{C}$ implies that $\alpha _{i}(x)\in \widehat{\text{Dom}}_{%
\mathbf{A}_{i}} \phi _{i}\left( \mathbf{C}\right) =\phi
_{i}\left( \mathbf{C}\right) $, whence $x\in \mathbf{C}$, and we conclude
that $\widehat{\text{Dom}}_{\mathbf{A}} \mathbf{C} \subseteq 
\mathbf{C}$.
\end{proof}

The next result is a straightforward consequence of the previous propositions.

\begin{corollary}\label{Cor2}
Let $\mathcal{V}$ be a variety of ordered algebras. Then epis are surjective
in $\mathcal{V}$ iff $\mathcal{V}$ has the special amalgamation property.
\end{corollary}

\section{Epis of ordered algebras} \label{SecEpis}

Our aim in this section is to prove that the varieties $\mathcal{F}$-\textbf{Oalg}$_{\leq
}^{0}$ have the special amalgamation property. We, however, first need to
prove few technical results. 
Let $\mathbf{A}_{i}$ for $i=1,2$ be ordered $\mathcal{F}$-algebras such that 
$A_1\cap A_2=\emptyset.$ As before, we consider the poset $(X,\leq _{X})$ where $X=A_1\dot{\cup} A_2$ and 
$\leq _{X}~=~\leq _{A_1}\dot{\cup}\leq _{A_2}$. In Section \ref{SecAmal} we defined the relations $\mathsf{R}_{A_{i}}$, for $i=1, 2$, on $T(X)$ by,
\[
\mathsf{R}_{A_i}=\left\{ \left( t(x_{1},\ldots ,x_{k}),t^{\mathbf{A}_i}(x_{1},\ldots ,x_{k})\right) \,|\,t(x_{1},\ldots ,x_{k})\in T(A_i)\right\}, 
\]
where $t^{\mathbf{A}_i}$ is the term function induced by
the term $t$ on $\mathbf{A}_i$.  
In the following Lemma we prove an important feature of the relation $\overset{\mathsf{R}_{A_{i}}^{-1}}{\longrightarrow}$ on $T(X)$.

\begin{lemma}
\label{Lemma 3} A scheme $p_{1}(u_{1})\overset{\mathsf{R}_{A_{i}}^{-1}}{\longrightarrow }p_{1}(v_{1})\preccurlyeq p_{2}(v_{2})$, for $i\in \{1,2\}$,
can always be rewritten as \[p_{1}(u_{1})\preccurlyeq p_{2}(u_{2})\overset{\mathsf{R}_{A_{i}}^{-1}}{\longrightarrow }p_{2}(v_{2}).\]
\end{lemma}

\begin{proof}
We shall suppose that $i=1$. The case $i=2$ is handled similarly. By the definition of the relation $\mathsf{R}_{A_{1}}$, we have 
$p_{1}(u_{1})\overset{\mathsf{R}_{A_{1}}^{-1}}{\longrightarrow }p_{1}(v_{1})$ iff $(u_{1}, v_{1})\in \mathsf{R}_{A_{1}}^{-1}$ and $p_1$ is an 
$X\cup \mathcal{F}_0$ - cfr translation. This means that 
\[(u_{1}, v_{1}) = (t^{\mathbf{A}_{1}}(a_{1},\ldots , a_{k}), t(a_{1},\ldots , a_{k})),\]
where $t(a_{1},\ldots , a_{k})\in T(A_1)$, and 
\[ p_1(u) = t_{1}(x_{1},\ldots,x_{l-1}, u, x_{l+1},\ldots ,x_{n}),\]
where the $n$-ary term $t_{1}(z_{1},\ldots ,z_{n})\in T_{\mathcal{F}}^{reg}(V)$ is constant free, and 
$x_{1},\ldots,x_{l-1}, x_{l+1},\ldots ,x_{n}\in X\cup \mathcal{F}_{0}$.
Without losing generality we can assume that $var(t)=(a_{1},\ldots , a_{k})$ and that $Tree(t)$ has $s$ leaves. 
Then by Lemma \ref{t=regt} there exist a constant free regular term $\overline{t}\left( z_{1},\ldots ,z_{s}\right)$ such that
\[
t\left( a_{1},\ldots , a_{k}\right) = \overline{t}\left( \overline{a}_{1},\ldots ,\overline{a}_{s}\right) 
\text{,}
\]
where $\overline{a}_{i} \in \{a_{1}, \ldots, a_{k}\} \cup \mathcal{F}_{0}$ for $i=1, \ldots, s$.
One may therefore write:
\begin{eqnarray*}
p_{1}(u_{1}) &=& t_{1}(x_{1},\dots ,x_{l-1}, \overline{t}^{\mathbf{A}_{1}}\left( \overline{a}_{1},\ldots ,\overline{a}_{s}\right), 
x_{l+1},\ldots , x_{n})\text{,} \\
p_{1}(v_{1}) &=& t_{1}(x_{1},\dots ,x_{l-1}, \overline{t}\left( \overline{a}_{1},\ldots ,\overline{a}_{s}\right), 
x_{l+1},\ldots , x_{n})
\text{.}
\end{eqnarray*}
The tree of the term $t_{1}(x_{1},\dots ,x_{l-1}, \overline{t}\left( \overline{a}_{1},\ldots ,\overline{a}_{s}\right), 
x_{l+1},\ldots ,x_{n})$ has $n+s-1$ leaves, where
 \[leaf(p_{1}(v_{1}))[i]=\begin{cases}
            x_i               &,\; \text{if}\,\, 1\leq i\leq l-1  \\
            \overline{a}_{i-l+1}     &,\; \text{if}\,\,  l\leq i\leq l+s-1 \\
            x_{i-s+1}    &,\; \text{if}\,\, l+s\leq i\leq n+s-1.
          \end{cases}
 \]         
Next, using the definition of $\preccurlyeq $, we may assert that
\begin{equation*}
p_{2}(v_{2})=t_{1}(y_{1},\dots ,y_{l-1}, \overline{t}\left( y_{l},\ldots , y_{l+s-1}\right), 
y_{l+s},\ldots ,y_{n+s-1})
\text{,}
\end{equation*}%
where $y_{1},\ldots , y_{n+s-1}\in X\cup \mathcal{F}_{0}$ and 
\begin{eqnarray}
x_i \leq_{X} y_i  \quad &\text{for}& \quad 1\leq i\leq l-1  \\
\overline{a}_{i-l+1} \leq_{A_{1}} y_i \quad &\text{for}& \quad l\leq i\leq l+s-1  \label{leafa} \\
x_{i-s+1} \leq_{X} y_i \quad &\text{for}& \quad  l+s\leq i\leq n+s-1.
\end{eqnarray}
The inequalities (\ref{leafa}) imply
\begin{equation*}
\overline{t}^{\mathbf{A}_{1}}(\overline{a}_{1},\ldots ,\overline{a}_{s})\leq _{A_{1}}
\overline{t}^{\mathbf{A}_{1}}(y_{l},\ldots , y_{l+s-1}).
\end{equation*}
So we can write:
\begin{eqnarray*}
p_{1}(u_{1}) &=&
t_{1}(x_{1},\dots ,x_{l-1}, \overline{t}^{\mathbf{A}_{1}}\left( \overline{a}_{1},\ldots ,\overline{a}_{s}\right), 
x_{l+1},\ldots , x_{n}) \\
                 &\preccurlyeq &t_{1}(y_{1},\dots ,y_{l-1}, \overline{t}^{\mathbf{A}_{1}}(y_{l},\ldots , y_{l+s-1}), 
y_{l+s},\ldots , y_{n+s-1})=p_2(u_2)\\
                 &\overset{\mathsf{R}_{A_{1}}^{-1}}{\longrightarrow }& 
 t_{1}(y_{1},\dots ,y_{l-1}, \overline{t}(y_{l},\ldots , y_{l+s-1}), y_{l+s},\ldots , y_{n+s-1})=p_2(v_2)\text{,}
\end{eqnarray*}
where $u_{2}=\overline{t}^{\mathbf{A}_{1}}(y_{l},\ldots , y_{l+s-1})$, $v_{2}=\overline{t}(y_{l},\ldots , y_{l+s-1})$ and 
$p_2(u) = t_{1}(y_{1},\ldots, y_{l-1}, u, y_{l+1},\ldots ,y_{n}).$ 
This completes the proof.
\end{proof}

\noindent One can also prove the following lemma in a dual manner.

\begin{lemma}
\label{Lemma 4}A scheme  $p_{1}(u_{1})\preccurlyeq p_{2}(u_{2})\overset{%
\mathsf{R}_{A_{i}}}{\longrightarrow }p_{2}(v_{2})$, $i\in \{1,2\}$, can always be
rewritten as \[p_{1}(u_{1})\overset{\mathsf{R}_{A_{i}}}{\longrightarrow }%
p_{1}(v_{1})\preccurlyeq p_{2}(v_{2}).\]
\end{lemma}

\noindent Next, let us consider a scheme of the form,
\[
p_{1}(u_{1}^{\prime })\preccurlyeq p_{2}(u_{2})\overset{
\mathsf{H}^{\prime }\dot{\cup}\mathsf{H}^{\prime -1}}{\longrightarrow }
p_{2}(u_{2}^{\prime })\preccurlyeq \cdots \preccurlyeq p_{n-1}(u_{n-1})
\overset{\mathsf{H}^{\prime }\dot{\cup}\mathsf{H}^{\prime -1}}{
\longrightarrow }p_{n-1}(u_{n-1}^{\prime })\preccurlyeq p_{n}(u_{n})
\text{,}
\]
Note that the trees of $p_{i}(u_{i})$ and $p_{j}(u_{j}^{\prime })$, $2\leq
i\leq n$, $1\leq j\leq n-1$, have the same number of leaves, say $m$. This
allows a representation of the above sequence by a rectangular grid with $m$
columns $c_{i}$, $1\leq i\leq m$,
\begin{equation}
\begin{tabular}{l|llllll|}
\cline{2-7}
& $c_{1}$ & \multicolumn{1}{|l}{$c_{2}$} & \multicolumn{1}{|l}{} & 
\multicolumn{1}{|l}{$c_{\alpha }$} & \multicolumn{1}{|l}{} & 
\multicolumn{1}{|l|}{$c_{m}$} \\ \cline{2-7}
$\ \ \ \ \ \ p_{1}(u_{1}^{\prime })$ $\ \ :$ & $x_{11}^{\prime }$ & 
\multicolumn{1}{|l}{$x_{12}^{\prime }$} & \multicolumn{1}{|l}{$\cdots $} & 
\multicolumn{1}{|l}{$x_{1\alpha }^{\prime }$} & \multicolumn{1}{|l}{$\cdots $
} & \multicolumn{1}{|l|}{$x_{1m}^{\prime }$} \\ \cline{2-7}
$\ \ \ \ \ \ p_{2}(u_{2})$ $\ \ :$ & $x_{21}$ & \multicolumn{1}{|l}{$x_{22}$}
& \multicolumn{1}{|l}{$\cdots $} & \multicolumn{1}{|l}{$x_{2\alpha }$} & 
\multicolumn{1}{|l}{$\cdots $} & \multicolumn{1}{|l|}{$x_{2m}$} \\ 
\cline{2-7}
$\ \ \ \ \ \ p_{2}(u_{2}^{\prime })$ $\ \ :$ & $x_{21}^{\prime }$ & 
\multicolumn{1}{|l}{$x_{22}^{\prime }$} & \multicolumn{1}{|l}{$\cdots $} & 
\multicolumn{1}{|l}{$x_{2\alpha }^{\prime }$} & \multicolumn{1}{|l}{$\cdots $
} & \multicolumn{1}{|l|}{$x_{2m}^{\prime }$} \\ \cline{2-7}
&  &  &  & $\vdots $ &  &  \\ \cline{2-7}
$p_{n-1}(u_{n-1})$ $\ \ :$ \ \  & $x_{n-1,1}$ & \multicolumn{1}{|l}{$
x_{n-1,2}$} & \multicolumn{1}{|l}{$\cdots $} & \multicolumn{1}{|l}{$
x_{n-1,\alpha }$} & \multicolumn{1}{|l}{$\cdots $} & \multicolumn{1}{|l|}{$
x_{n-1,m}$} \\ \cline{2-7}
$p_{n-1}(u_{n-1}^{\prime })$ $\ \ :$ \ \  & $x_{n-1,1}^{\prime }$ & 
\multicolumn{1}{|l}{$x_{n-1,2}^{\prime }$} & \multicolumn{1}{|l}{$\cdots $}
& \multicolumn{1}{|l}{$x_{n-1,\alpha }^{\prime }$} & \multicolumn{1}{|l}{$
\cdots $} & \multicolumn{1}{|l|}{$x_{n-1,m}^{\prime }$} \\ \cline{2-7}
$\quad \,\ \ p_{n}(u_{n})$ $\ \ :$ & $x_{n1}$ & \multicolumn{1}{|l}{$x_{n2}$}
& \multicolumn{1}{|l}{$\cdots $} & \multicolumn{1}{|l}{$x_{n\alpha }$} & 
\multicolumn{1}{|l}{$\cdots $} & \multicolumn{1}{|l|}{$x_{nm}$} \\ 
\cline{2-7}
\end{tabular}
\text{ ,}
\label{grid1}
\end{equation}%
where the sequences $(x_{i1}^{\prime },\ldots ,x_{im}^{\prime })$, $1\leq
i\leq n-1$, and $(x_{j1},\ldots ,x_{jm})$, $2\leq j\leq n$, that form rows of the above grid, 
give the labeling (first to last, left to right) of the leaves of 
Tree$(p_{i}(u_{i}^{\prime}))$ and Tree$(p_{j}(u_{j}^{\prime }))$ respectively. 
We must also have $u_{i}^{\prime }=t_{i}^{\prime }( x_{ip_{i}}^{\prime },\ldots
,x_{iq_{i}}^{\prime }) $, for some $1\leq p_{i}\leq q_{i}\leq m$, and $u_{j}=t_{j}(x_{jp_{j}},\ldots ,x_{jq_{j}})$, for some $1\leq p_{j}\leq q_{j}\leq m$.

\begin{definition}
\upshape Let $u_{1}^{\prime }=t_{1}^{\prime }\left( x_{1p_{1}}^{\prime
},\ldots ,x_{1q_{1}}^{\prime }\right) $ and $u_{n}=t_{n}(x_{np_{n}},\ldots
,x_{nq_{n}})$. We say that $u_{n}$ covers $u_{1}^{\prime }$ if $p_{n}\leq
p_{1}$, $q_{1}\leq q_{n}$. Dually, one can say that $u_{1}^{\prime}$ covers 
$u_{n}$. If $p_{1}=p_{n}$, $q_{1}=q_{n}$, then $u_{1}^{\prime }$, $u_{n}$ are
said to cover each other properly.
\end{definition}

\begin{remark}
\upshape Because $p_{1}(u_{1}^{\prime })$ and $p_{n}(u_{n})$ have the same skeleton, we cannot have $u_{1}^{\prime}$ and $u_{n}$ partially overlap, i.e., 
the case $p_{1} < p_{n}$, $q_{1} < q_{n}$ and its dual cannot arise.
\end{remark}

\begin{definition}
\upshape Let the $k$-ary term $q(z_{1}, z_{2},\ldots z_{k})\in T_{\mathcal{F}}^{reg}(V)$ be constant free. Then we write 
\begin{equation*}
q(u_{1}, u_{2},\ldots u_{k})\overset{\mathsf{H}^{\prime }\cup \mathsf{H}^{\prime -1}\cup \mathsf{I}_{A_{1}\cup A_{2}}}{\rightrightarrows }
q(u_{1}^{\prime }, u_{2}^{\prime},\ldots u_{k}^{\prime })\text{,}
\end{equation*}
if $(u_{i}, u_{i}^{\prime })\in \mathsf{H}^{\prime }\dot{\cup}\mathsf{H}%
^{\prime -1}\dot{\cup}\mathsf{I}_{A_1\dot{\cup}A_2}$ for all $1\leq i\leq k$, where 
$\mathsf{I}_{A_1\dot{\cup}A_2}$ is the diagonal relation. 
\end{definition}

\begin{lemma}
\label{Lemma 5}Given a scheme 
\[
p_{1}(u_{1}^{\prime })\preccurlyeq p_{2}(u_{2})\overset{
\mathsf{H}^{\prime }\dot{\cup}\mathsf{H}^{\prime -1}}{\longrightarrow }
p_{2}(u_{2}^{\prime })\preccurlyeq \cdots \preccurlyeq p_{n-1}(u_{n-1})
\overset{\mathsf{H}^{\prime }\dot{\cup}\mathsf{H}^{\prime -1}}{
\longrightarrow }p_{n-1}(u_{n-1}^{\prime })\preccurlyeq p_{n}(u_{n})
\text{,}
\]
one may always rewrite it as
\begin{equation*}
 p_1(u_{1}^{\prime }) \preccurlyeq 
p_{2}^{\prime }(u_{2})\overset{\mathsf{H}^{\prime }\dot{\cup
}\mathsf{H}^{\prime -1}\dot{\cup}\mathsf{I}_{A_1\dot{\cup}A_2}}{
\rightrightarrows }p_{n-1}^{\prime }(u_{n-1}^{\prime })
\preccurlyeq p_{n}(u_{n})\text{.}
\end{equation*}
\end{lemma}

\begin{proof}
Considering a column $c_{\alpha}$, $1\leq \alpha \leq m$, in grid (\ref
{grid1}), we have three possibilities.

\noindent (i) If $(x_{j\alpha },x_{j\alpha }^{\prime })\in \mathsf{I}%
_{A_{1}\dot{\cup} A_{2}}$ for all $2\leq j\leq n-1$, then $x_{1\alpha }^{\prime
}\leq x_{n\alpha }$ and we can rewrite grid (\ref{grid1}) as%
\begin{equation}
\begin{tabular}{l|llllll|}
\cline{2-7}
& $c_{1}$ & \multicolumn{1}{|l}{$c_{2}$} & \multicolumn{1}{|l}{$\cdots $} & 
\multicolumn{1}{|l}{$c_{\alpha }$} & \multicolumn{1}{|l}{$\cdots $} & 
\multicolumn{1}{|l|}{$c_{m}$} \\ \cline{2-7}
$\ \ \ \ \ \ p_{1}(u_{1}^{\prime })$ $\ \ :$ & $x_{11}^{\prime }$ & 
\multicolumn{1}{|l}{$x_{12}^{\prime }$} & \multicolumn{1}{|l}{$\cdots $} & 
\multicolumn{1}{|l}{$x_{1\alpha }^{\prime }$} & \multicolumn{1}{|l}{$\cdots $%
} & \multicolumn{1}{|l|}{$x_{1m}^{\prime }$} \\ \cline{2-7}
$\ \ \ \ \ \ p_{2}^{\prime }(u_{2})$ $\ \ :$ & $x_{21}$ & 
\multicolumn{1}{|l}{$x_{22}$} & \multicolumn{1}{|l}{$\cdots $} & 
\multicolumn{1}{|l}{$x_{n\alpha }$} & \multicolumn{1}{|l}{$\cdots $} & 
\multicolumn{1}{|l|}{$x_{2m}$} \\ \cline{2-7}
$\ \ \ \ \ \ p_{2}^{\prime }(u_{2}^{\prime })$ $\ \ :$ & $x_{21}^{\prime }$
& \multicolumn{1}{|l}{$x_{22}^{\prime }$} & \multicolumn{1}{|l}{$\cdots $} & 
\multicolumn{1}{|l}{$x_{n\alpha }$} & \multicolumn{1}{|l}{$\cdots $} & 
\multicolumn{1}{|l|}{$x_{2m}^{\prime }$} \\ \cline{2-7}
&  &  &  & $\vdots $ &  &  \\ \cline{2-7}
$p_{n-1}^{\prime }(u_{n-1})$ $\ \ :$ \ \  & $x_{n-1,1}$ & 
\multicolumn{1}{|l}{$x_{n-1,2}$} & \multicolumn{1}{|l}{$\cdots $} & 
\multicolumn{1}{|l}{$x_{n\alpha }$} & \multicolumn{1}{|l}{$\cdots $} & 
\multicolumn{1}{|l|}{$x_{n-1,m}$} \\ \cline{2-7}
$p_{n-1}^{\prime }(u_{n-1}^{\prime })$ $\ \ :$ \ \  & $x_{n-1,1}^{\prime }$
& \multicolumn{1}{|l}{$x_{n-1,2}^{\prime }$} & \multicolumn{1}{|l}{$\cdots $}
& \multicolumn{1}{|l}{$x_{n\alpha }$} & \multicolumn{1}{|l}{$\cdots $} & 
\multicolumn{1}{|l|}{$x_{n-1,m}^{\prime }$} \\ \cline{2-7}
$\quad \,\ \ p_{n}(u_{n})$ $\ \ :$ & $x_{n1}$ & \multicolumn{1}{|l}{$x_{n2}$}
& \multicolumn{1}{|l}{$\cdots $} & \multicolumn{1}{|l}{$x_{n\alpha }$} & 
\multicolumn{1}{|l}{$\cdots $} & \multicolumn{1}{|l|}{$x_{nm}$} \\ 
\cline{2-7}
\end{tabular}
\text{ ,}
\label{grid2}
\end{equation}
where $p_{j}^{\prime }(u_{j})$ and $p_{j}^{\prime }(u_{j}^{\prime })$ in (\ref{grid2}), for $2\leq j\leq n-1$, are
the translations obtained from $p_{j}(u_{j})$ and $p_{j}(u_{j}^{\prime })$ in (\ref{grid1}) respectively, by relabeling the leaves
in column $c_{\alpha }$.\smallskip

\noindent (ii) If we have an even number of relations $x_{j\alpha }\overset{\mathsf{H}^{\prime }\cup \mathsf{H}
^{\prime -1}}{\longrightarrow }x_{j\alpha }^{\prime }$, $j=2,\ldots,n-1,$
in column $c_{\alpha }$, then both $x_{1\alpha }$ and $x_{n\alpha }$ in grid (\ref{grid1}) are in
the same $A_{i}$, $i\in \{1,2\}$. If $x_{1\alpha }$ and $%
x_{n\alpha }$ are in $A_{1}$, then we may apply $\nu ^{-1}$ across
the inequalities $x_{j\alpha }^{\prime }\leq _{A_2}x_{j+1,\alpha }$, where $
j\in \{2,\ldots ,n-1\}$. And if $x_{1\alpha }$ and $x_{n\alpha }$ are in $A_{2}$, then we can apply $\nu $ across the inequalities 
$x_{j\alpha }^{\prime }\leq _{A_1}x_{j+1,\alpha }$, $k\in \{2,\ldots ,n-1\}$.
As a result $c_{\alpha }$ is transformed into column that we considered in
case (i).

\noindent (iii) If we have an odd number of relations $x_{j\alpha }\overset{\mathsf{H}^{\prime }\cup \mathsf{H}
^{\prime -1}}{\longrightarrow }x_{j\alpha }^{\prime }$, $j=2,\ldots,n-1,$ in column $c_{\alpha }$, 
then $x_{1\alpha }$ and $x_{n\alpha }$ belong to different $A_{i}$, $i\in \{1,2\}$. Take
\[
k=\max \{j\in \mathbb{N} \,|\, (x_{j\alpha },x_{j\alpha }^{\prime })\in \mathsf{H}^{\prime }\cup \mathsf{H}^{\prime -1}\}\text{.}
\]
Now, if $(x_{k\alpha },x_{k\alpha }^{\prime })\in \mathsf{H}^{\prime }$,
then we apply $\nu^{-1}$ across the inequalities 
$x_{j\alpha }\leq_{A_2}x_{j+1,\alpha }^{\prime }$, $j< k$. By the argument used in case (i),
this gives $x_{1\alpha}^{\prime} \leq _{A_1} x_{k\alpha }\,\, \mathsf{H}^{\prime }\,\, x_{k\alpha }^{\prime } \leq _{A_2} x_{n\alpha }$. 
Dually, if $(x_{k\alpha
},x_{k\alpha }^{\prime })\in \mathsf{H}^{\prime -1}$, then applying $\nu $
across the inequalities $x_{j\alpha }^{\prime }\leq _{1}x_{j+1,\alpha }$, 
$ j < k$, we get $x_{1\alpha }^{\prime} \leq _{A_2} x_{k\alpha }\,\, \mathsf{H}^{\prime -1}\,\, x_{k\alpha }^{\prime } \leq _{A_1} x_{n\alpha }$. In either cases grid (\ref{grid1}) may be rewritten as
\begin{equation}
\begin{tabular}{l|llllll|}
\cline{2-7}
& $c_{1}$ & \multicolumn{1}{|l}{$c_{2}$} & \multicolumn{1}{|l}{$\cdots $} & 
\multicolumn{1}{|l}{$c_{\alpha }$} & \multicolumn{1}{|l}{$\cdots $} & 
\multicolumn{1}{|l|}{$c_{m}$} \\ \cline{2-7}
$\ \ \ \ \ \ p_{1}(u_{1}^{\prime })$ $\ \ :$ & $x_{11}^{\prime }$ & 
\multicolumn{1}{|l}{$x_{12}^{\prime }$} & \multicolumn{1}{|l}{$\cdots $} & 
\multicolumn{1}{|l}{$x_{1\alpha }^{\prime }$} & \multicolumn{1}{|l}{$\cdots $%
} & \multicolumn{1}{|l|}{$x_{1m}^{\prime }$} \\ \cline{2-7}
$\ \ \ \ \ \ p_{2}^{\prime }(u_{2})$ $\ \ :$ & $x_{21}$ & 
\multicolumn{1}{|l}{$x_{22}$} & \multicolumn{1}{|l}{$\cdots $} & 
\multicolumn{1}{|l}{$x_{k\alpha }$} & \multicolumn{1}{|l}{$\cdots $} & 
\multicolumn{1}{|l|}{$x_{2m}$} \\ \cline{2-7}
$\ \ \ \ \ \ p_{2}^{\prime }(u_{2}^{\prime })$ $\ \ :$ & $x_{21}^{\prime }$
& \multicolumn{1}{|l}{$x_{22}^{\prime }$} & \multicolumn{1}{|l}{$\cdots $} & 
\multicolumn{1}{|l}{$x_{k\alpha }$} & \multicolumn{1}{|l}{$\cdots $} & 
\multicolumn{1}{|l|}{$x_{2m}^{\prime }$} \\ \cline{2-7}
&  &  &  & $\vdots $ &  &  \\ \cline{2-7}
$\ \ \ \ \ \ p_{k}^{\prime }(u_{k})$ $\ \ :$ & $x_{k1}$ & 
\multicolumn{1}{|l}{$x_{k2}$} & \multicolumn{1}{|l}{$\cdots $} & 
\multicolumn{1}{|l}{$x_{k\alpha }$} & \multicolumn{1}{|l}{$\cdots $} & 
\multicolumn{1}{|l|}{$x_{km}$} \\ \cline{2-7}
$\ \ \ \ \ \ p_{k}^{\prime }(u_{k}^{\prime })$ $\ \ :$ & $x_{k1}^{\prime }$
& \multicolumn{1}{|l}{$x_{k2}^{\prime }$} & \multicolumn{1}{|l}{$\cdots $} & 
\multicolumn{1}{|l}{$x_{k\alpha }^{\prime }$} & \multicolumn{1}{|l}{$\cdots $%
} & \multicolumn{1}{|l|}{$x_{km}^{\prime }$} \\ \cline{2-7}
&  &  &  & $\vdots $ &  &  \\ \cline{2-7}
$p_{n-1}^{\prime }(u_{n-1})$ $\ \ :$ \ \  & $x_{n-1,1}$ & 
\multicolumn{1}{|l}{$x_{n-1,2}$} & \multicolumn{1}{|l}{$\cdots $} & 
\multicolumn{1}{|l}{$x_{k\alpha }^{\prime }$} & \multicolumn{1}{|l}{$\cdots $%
} & \multicolumn{1}{|l|}{$x_{n-1,m}$} \\ \cline{2-7}
$p_{n-1}^{\prime }(u_{n-1}^{\prime })$ $\ \ :$ \ \  & $x_{n-1,1}^{\prime }$
& \multicolumn{1}{|l}{$x_{n-1,2}^{\prime }$} & \multicolumn{1}{|l}{$\cdots $}
& \multicolumn{1}{|l}{$x_{k\alpha }^{\prime }$} & \multicolumn{1}{|l}{$%
\cdots $} & \multicolumn{1}{|l|}{$x_{n-1,m}^{\prime }$} \\ \cline{2-7}
$\quad \,\ \ p_{n}(u_{n})$ $\ \ :$ & $x_{n1}$ & \multicolumn{1}{|l}{$x_{n2}$}
& \multicolumn{1}{|l}{$\cdots $} & \multicolumn{1}{|l}{$x_{n\alpha }$} & 
\multicolumn{1}{|l}{$\cdots $} & \multicolumn{1}{|l|}{$x_{nm}$} \\ 
\cline{2-7}
\end{tabular}
\text{ ,}
\label{grid3}
\end{equation}
where $p_{j}^{\prime }(u_{j})$ and $p_{j}^{\prime }(u_{j}^{\prime })$ in (\ref{grid3}), for $2\leq j\leq n-1$, are
the translations obtained from $p_{j}(u_{j})$ and $p_{j}(u_{j}^{\prime })$ in (\ref{grid1}) respectively, by relabeling the leaves
in column $c_{\alpha }$.\smallskip

Applying one of the processes used in cases (i), (ii) and (iii), as
applicable, to each of its columns, we can clearly contract grid (\ref{grid1}%
) to,%
\begin{equation*}
\begin{tabular}{l|l|l|l|l|l|l|}
\cline{2-7}
& $c_{1}$ & $c_{2}$ &  & $c_{\alpha }$ &  & $c_{m}$ \\ \cline{2-7}
$\ \ \ \ \ \ p_{1}(u_{1}^{\prime })$ $\ \ :$ & $x_{11}^{\prime }$ & $%
x_{12}^{\prime }$ & $\cdots $ & $x_{1\alpha }^{\prime }$ & $\cdots $ & $%
x_{1m}^{\prime }$ \\ \cline{2-7}
$\ \ \ \ \ \ p_{2}^{\prime }(u_{2})$ $\ \ :$ & $x_{21}$ & $x_{22}$ & $\cdots 
$ & $x_{2\alpha }$ & $\cdots $ & $x_{2m}$ \\ \cline{2-7}
$p_{n-1}^{\prime }(u_{n-1}^{\prime })$ $\ \ :$ \ \  & $x_{n-1,1}^{\prime }$
& $x_{n-1,2}^{\prime }$ & $\cdots $ & $x_{n-1,\alpha }^{\prime }$ & $\cdots $
& $x_{n-1,m}^{\prime }$ \\ \cline{2-7}
$\quad \,\ \ p_{n}(u_{n})$ $\ \ :$ & $x_{n1}$ & $x_{n2}$ & $\cdots $ & $%
x_{n\alpha }$ & $\cdots $ & $x_{nm}$ \\ \cline{2-7}
\end{tabular}%
\text{ ,}
\end{equation*}%
where $p_{2}^{\prime }(u_{2})\overset{\mathsf{H}^{\prime }\cup \mathsf{H}%
^{\prime -1}\cup \mathsf{I}_{A_1\cup A_2}}{\rightrightarrows }p_{n-1}^{\prime
}(u_{n-1}^{\prime })$. 
So, we can write
\begin{equation*}
p_{1}(u_{1}^{\prime })\preccurlyeq p_{2}^{\prime }(u_{2})
\overset{\mathsf{H}^{\prime }\cup \mathsf{H}^{\prime -1}\cup \mathsf{I}
_{A_1\cup A_2}}{\rightrightarrows }p_{n-1}^{\prime }(u_{n-1}^{\prime
})\preccurlyeq p_{n}(x_{n})
\text{.}
\end{equation*}
This completes the proof.
\end{proof}

\newpage

\begin{theorem}\label{main}
$\mathcal{F}$-\textbf{Oalg}$_{\leq }^{0}$ have special amalgamation property. 
\end{theorem}

\begin{proof}
Let $(\mathbf{C};\mathbf{A}_1,\mathbf{A}_2 ;\phi _{1},\phi _{2})$ be a special amalgam in a variety $\mathcal{F}$-\textbf{Oalg}$_{\leq }^{0}$.
Let $x\in A_{1}$, $y\in A_{2}$ be such that $\mu _{1}\left( x\right) =\mu
_{2}(y)$ in $\mathbf{A}_{1}\amalg _{\mathbf{C}}\mathbf{A}_{2}$. Then there
exist schemes
\begin{equation}
\begin{tabular}{l}
$x\preccurlyeq p_{1}(x_{1}) \overset{\widehat{\mathsf{H}}\dot{\cup}\widehat{\mathsf{H}}^{-1}}{\longrightarrow }
p_{1}(x_{1}^{\prime })\preccurlyeq p_{2}(x_{2})
\overset{\widehat{\mathsf{H}}\dot{\cup}\widehat{\mathsf{H}}^{-1}}{\longrightarrow }
p_{2}(x_{2}^{\prime })\preccurlyeq \cdots \preccurlyeq
p_{n}(x_{n})
\overset{\widehat{\mathsf{H}}\dot{\cup}\widehat{\mathsf{H}}^{-1}}{\longrightarrow }
 p_{n}(x_{n}^{\prime })\preccurlyeq y$ \\ 
\\ 
$y\preccurlyeq q_{1}(y_{1})  \overset{\widehat{\mathsf{H}}\dot{\cup}\widehat{\mathsf{H}}^{-1}}{\longrightarrow }
q_{1}(y_{1}^{\prime })\preccurlyeq q_{2}(y_{2})
 \overset{\widehat{\mathsf{H}}\dot{\cup}\widehat{\mathsf{H}}^{-1}}{\longrightarrow }
q_{2}(y_{2}^{\prime })\preccurlyeq \cdots \preccurlyeq 
q_{m}(y_{m})
 \overset{\widehat{\mathsf{H}}\dot{\cup}\widehat{\mathsf{H}}^{-1}}{\longrightarrow }
 q_{m}(y_{m}^{\prime })\preccurlyeq x.$
\end{tabular}
\label{Sequences}
\end{equation}

\textbf{Case 1.} Let all the trees representing the terms $p_{i}(x_{i})$, $p_{i}(x_{i}^{\prime })$, $q_{j}(y_{j})$, $q_{j}(y_{j}^{\prime })$ in 
(\ref{Sequences}) have only one node. Then these nodes must belong to $A_1 \cup A_2  \cup \mathcal{F}_{0}^{\mathbf{A}_1} \cup \mathcal{F}_{0}^{\mathbf{A}_2}$, 
and therefore $\overset{\widehat{\mathsf{H}}\dot{\cup}\widehat{\mathsf{H}}^{-1}}{\longrightarrow }=
\overset{\mathsf{H}^{\prime } \dot{\cup} \mathsf{H}^{\prime -1} \cup \mathsf{I}
_{A_1\cup A_2} }{\longrightarrow }.
$
Consequently, we have from (\ref{Sequences})
\begin{equation*}
x\leq_{A_1} \phi _{1}(u_{1}) \overset{\mathsf{H}^{\prime}}{\longrightarrow }
\phi _{2}(u_{1})\leq_{A_2}
\phi_{2}(u_{2}) \overset{\mathsf{H}^{\prime -1}}{\longrightarrow }
\phi _{1}(u_{2})\leq_{A_1} \cdots \leq_{A_1} 
\phi_{1}(u_{n}) \overset{\mathsf{H}^{\prime}}{\longrightarrow }
\phi _{2}(u_{n})\leq_{A_2} y\text{,}
\end{equation*}%
\begin{equation*}
y\leq_{A_2} \phi _{2}(v_{1}) \overset{\mathsf{H}^{\prime -1}}{\longrightarrow }
\phi _{1}(v_{1})\leq_{A_1} \phi _{1}(v_{2})
\overset{\mathsf{H}^{\prime}}{\longrightarrow }
\phi _{2}(v_{2})\leq_{A_2} \cdots \leq_{A_2} \phi_{2}(v_{m})
\overset{\mathsf{H}^{\prime -1}}{\longrightarrow } 
\phi _{1}(v_{m})\leq_{A_1} x\text{,}
\end{equation*}%
where $u_{i}$, $v_{j}\in C$, $1\leq i\leq n$, $1\leq j\leq m$. Applying $\nu ^{-1}$ across the
inequalities in $\mathbf{A}_{2}$, we get
\begin{equation*}
x\leq _{A_1}\phi _{1}(u_{1})\leq _{A_1}\phi _{1}(u_{2})\leq _{A_1}\cdots \leq
_{A_1}\phi _{1}(u_{n})\leq _{A_1}\nu ^{-1}(y)
\end{equation*}%
\begin{equation*}
\nu ^{-1}(y)\leq _{A_1}\phi _{1}(v_{1})\leq _{A_1}\phi _{1}(v_{2})\leq
_{A_1}\cdots \leq _{A_1}\phi _{1}(v_{m})\leq _{A_1}x\text{.}
\end{equation*}%
This implies that $x=\phi _{1}(u_{1})=\phi _{1}(u_{2})=\cdots =\phi
_{1}(u_{n})=\nu ^{-1}(y)=\phi _{1}(v_{1})=\phi _{1}(v_{2})=\cdots =\phi
_{1}(v_{m})$. Since $\phi _{1}$ is an order-embedding we can further assert
that $u_{1}=u_{2}=\cdots =u_{n}=v_{1}=v_{2}=\cdots =v_{m}=z\in C$, say. So, we have 
\begin{equation*}
x=\phi _{1}(z)\text{, }y=\phi _{2}(z)\text{.}
\end{equation*}%
Thus $(\mathbf{C};\mathbf{A}_{1},\mathbf{A}_{2};\phi _{1},\phi _{2})$ is
embeddable. 

\textbf{Case 2.} Let some of the trees representing the terms $p_{i}(x_{i})$%
, $p_{i}(x_{i}^{\prime })$, $q_{j}(y_{j})$, $q_{j}(y_{j}^{\prime })$, $1\leq
i\leq n$, $1\leq j\leq m$, in (\ref{Sequences}) have more than one node.
We show that this case can be reduced to Case 1. To this end, 
we shall only consider the scheme from $x$ to $y$, viz.
\begin{equation}
x\preccurlyeq p_{1}(x_{1}) \overset{\widehat{\mathsf{H}}\dot{\cup}\widehat{\mathsf{H}}^{-1}}{\longrightarrow }
p_{1}(x_{1}^{\prime })\preccurlyeq p_{2}(x_{2})
\overset{\widehat{\mathsf{H}}\dot{\cup}\widehat{\mathsf{H}}^{-1}}{\longrightarrow }
p_{2}(x_{2}^{\prime })\preccurlyeq \cdots \preccurlyeq
p_{n}(x_{n})
\overset{\widehat{\mathsf{H}}\dot{\cup}\widehat{\mathsf{H}}^{-1}}{\longrightarrow }
 p_{n}(x_{n}^{\prime })\preccurlyeq y
\text{.}  \label{Sequence}
\end{equation}%
By analogy our argument will also work for the scheme
from $y$ to $x$. Because $x$ and $y$ both have trees with just one node, between $\overset{\mathsf{H}}{\longrightarrow }$ and 
$\overset{\mathsf{H}^{-1}}{\longrightarrow }$
we shall first encounter relations of type
\begin{equation*}
p_{i}(x_{i})\overset{\mathsf{H}^{-1}}{\longrightarrow }p_{i}(x_{i}^{\prime })
\text{,}
\end{equation*}%
while writing scheme (\ref{Sequence}); note that $p_{j}(x_{j})\overset{\mathsf{H}}{\longrightarrow } p_{j}(x_{j}^{\prime })$
reduces to $p_{j}(x_{j})=p_{j}(x_{j})$ if $x_j\in \mathcal{F}_{0}^{\mathbf{A}_{1}}\cup \mathcal{F}_{0}^{\mathbf{A}_{2}}.$

By a similar token, moving backwards
from $y$ to $x$ in (\ref{Sequence}) we shall first encounter
relations of type%
\begin{equation*}
p_{i}(x_{i})\overset{\mathsf{H}}{\longrightarrow }p_{i}(x_{i}^{\prime })%
\text{.}
\end{equation*}%
This implies that there exists a segment of the following type in the scheme 
under consideration,
\begin{equation*}
p_{s}(x_{s})\overset{\mathsf{H}^{-1}}{\longrightarrow }p_{s}(x_{s}^{\prime
})\preccurlyeq p_{s+1}(x_{s+1})\overset{\mathsf{H}^{^{\prime }}\cup \mathsf{H%
}^{\prime -1}}{\longrightarrow }p_{s+1}(x_{s+1}^{\prime })\preccurlyeq
p_{s+2}(x_{s+2})\overset{\mathsf{H}^{^{\prime }}\cup \mathsf{H}^{\prime -1}}{%
\longrightarrow }
\end{equation*}%
\begin{equation}
\cdots \overset{\mathsf{H}^{^{\prime }}\cup \mathsf{H}^{\prime -1}}{
\longrightarrow }p_{s+t-1}(x_{s+t-1}^{\prime })\preccurlyeq p_{s+t}(x_{s+t})
\overset{\mathsf{H}}{\longrightarrow }p_{s+t}(x_{s+t}^{\prime })\text{
,\medskip }
\label{seg1}
\end{equation}
such that for all $j<s$, we have no relations of the form 
 $p_{j}(x_{j})\overset{\mathsf{H}}{\longrightarrow }p_{j}(x_{j}^{\prime })$.\medskip

\noindent Using Lemma (\ref{Lemma 5}) we can rewrite the segment (\ref{seg1}) as
\begin{equation}
p_{s}(x_{s}) \overset{\mathsf{H}^{-1}}{\longrightarrow } p_{s}(x_{s}^{\prime }) \preccurlyeq
p_{s+1}^{\prime }(x_{s+1})
\overset{\mathsf{H}^{\prime }\cup \mathsf{H}^{\prime -1}\cup \mathsf{I}%
_{A_{1}\cup A_{2}}}{\rightrightarrows }
p_{s+t-1}^{\prime }(x_{s+t-1}^{\prime }) \preccurlyeq p_{s+t}(x_{s+t})
\overset{\mathsf{H}}{\longrightarrow }p_{s+t}(x_{s+t}^{\prime})\text{.}  
\label{segment1}
\end{equation}%
By exhausting all the possibilities we shall first show that it is always
possible to replace segment (\ref{segment1}) by one in which either $\mathsf{%
H}^{-1}$ and $\mathsf{H}$ are either swapped or at least one of the symbols $%
\mathsf{H}^{-1}$ and $\mathsf{H}$ vanishes. \medskip

\noindent \textbf{Case 2a. } Assume that 
\begin{eqnarray*}
p_{s}(x_{s}) &=&t_{s}(a_{1},\ldots ,t^{\mathbf{A}_{i}}(a_{r},\ldots
,a_{r+s}),\ldots ,t^{\prime }(a_{r^{\prime }},\ldots ,a_{r^{\prime
}+s^{\prime }}),\ldots ,a_{m}) \\
p_{s}(x_{s}^{\prime }) &=&t_{s}(a_{1},\ldots ,t(a_{r},\ldots
,a_{r+s}),\ldots ,t^{\prime }(a_{r^{\prime }},\ldots ,a_{r^{\prime
}+s^{\prime }}),\ldots ,a_{m}) \\
p_{s+t}(x_{s+t}) &=&t_{s}(a_{1}^{\prime },\ldots ,t(a_{r}^{\prime },\ldots
,a_{r+s}^{\prime }),\ldots ,t^{\prime }(a_{r^{\prime }}^{\prime },\ldots
,a_{r^{\prime }+s^{\prime }}^{\prime }),\ldots ,a_{m}^{\prime }) \\
p_{s+t}(x_{s+t}^{\prime }) &=&t_{s}(a_{1}^{\prime },\ldots ,t(a_{r}^{\prime
},\ldots ,a_{r+s}^{\prime }),\ldots ,t^{\prime \mathbf{A}_{j}}(a_{r^{\prime
}}^{\prime },\ldots ,a_{r^{\prime }+s^{\prime }}^{\prime }),\ldots
,a_{m}^{\prime })\text{,}
\end{eqnarray*}%
where $i,j\in \{1,2\}$, $(a_{\alpha },a_{\alpha }^{\prime })\in \mathsf{H}%
^{\prime }\cup \mathsf{H}^{\prime -1}\cup \mathsf{I}_{A_{1}\cup A_{2}}$, $%
1\leq \alpha \leq m$, and%
\begin{eqnarray*}
x_{s} &=&t^{\mathbf{A}_{i}}(a_{r},\ldots ,a_{r+s}) \\
x_{s}^{\prime } &=&t(a_{r},\ldots ,a_{r+s}) \\
x_{s+t} &=&t^{\prime }(a_{r^{\prime }}^{\prime },\ldots ,a_{r^{\prime
}+s^{\prime }}^{\prime }) \\
x_{s+t}^{\prime } &=&t^{\prime \mathbf{A}_{j}}(a_{r^{\prime }}^{\prime
},\ldots ,a_{r^{\prime }+s^{\prime }}^{\prime })\text{.}
\end{eqnarray*}%
In this case we take,
\begin{eqnarray*}
p_{s}(x_{s}) &=&t_{s}(a_{1},\ldots ,t^{\mathbf{A}_{i}}(a_{r},\ldots,a_{r+s}),\ldots ,t^{\prime }(a_{r^{\prime }},\ldots ,a_{r^{\prime
}+s^{\prime }}),\ldots ,a_{m}) \\
p_{s}^{\prime }(x_{s+t}) &=&t_{s}(a_{1}^{\prime },\ldots ,t^{\mathbf{A}_{i}}(a_{r},\ldots ,a_{r+s}),\ldots ,t^{\prime }(a_{r^{\prime }}^{\prime
},\ldots ,a_{r^{\prime }+s^{\prime }}^{\prime }),\ldots ,a_{m}^{\prime }) \\
p_{s}^{\prime }(x_{s+t}^{\prime }) &=&t_{s}(a_{1}^{\prime },\ldots ,t^{
\mathbf{A}_{i}}(a_{r},\ldots ,a_{r+s}),\ldots ,t^{\prime \mathbf{A}
_{j}}(a_{r^{\prime }}^{\prime },\ldots ,a_{r^{\prime }+s^{\prime }}^{\prime
}),\ldots ,a_{m}^{\prime }) \\
p_{s+t}^{\prime}(x_{s}^{\prime }) &=&t_{s}(a_{1}^{\prime },\ldots
,t(a_{r},\ldots ,a_{r+s}),\ldots ,t^{\prime \mathbf{A}_{j}}(a_{r^{\prime
}}^{\prime },\ldots ,a_{r^{\prime }+s^{\prime }}^{\prime }),\ldots
,a_{m}^{\prime }) \\
p_{s+t}(x_{s+t}^{\prime }) &=&t_{s}(a_{1}^{\prime },\ldots ,t(a_{r}^{\prime
},\ldots ,a_{r+s}^{\prime }),\ldots ,t^{\prime \mathbf{A}_{j}}(a_{r^{\prime
}}^{\prime },\ldots ,a_{r^{\prime }+s^{\prime }}^{\prime }),\ldots
,a_{m}^{\prime })\text{.}
\end{eqnarray*}%
So, we can re-write segment (\ref{segment1}) as,%
\begin{equation*}
p_{s}(x_{s})\overset{\mathsf{H}^{\prime }\cup \mathsf{H}^{\prime -1}\cup 
\mathsf{I}_{A_{1}\cup A_{2}}}{\rightrightarrows }p_{s}^{\prime }(x_{s+t})
\overset{\mathsf{H}}{\longrightarrow }p_{s}^{\prime }(x_{s+t}^{\prime
})= p_{s+t}^{\prime }(x_{s})\overset{\mathsf{H}^{-1}}{\longrightarrow }
p_{s+t}^{\prime }(x_{s}^{\prime })\overset{\mathsf{H}^{\prime }\cup 
\mathsf{H}^{\prime -1}\cup \mathsf{I}_{A_{1}\cup A_{2}}}{\rightrightarrows }
p_{s+t}(x_{s+t}^{\prime })\text{,}\bigskip
\end{equation*}

The dual case when 
$
p_{s}(x_{s}) = t_{s}(a_{1},\ldots ,t^{\prime }(a_{r^{\prime }},\ldots ,a_{r^{\prime}+s^{\prime }}),\ldots ,t^{\mathbf{A}_{i}}(a_{r},\ldots ,a_{r+s}),
\ldots ,a_{m}) 
$
can be dealt with in a similar way.

\noindent \textbf{Case 2b. }If in (\ref{segment1}) $x_{s}^{\prime }$
properly covers $x_{s+t}$, then we we have%
\begin{eqnarray*}
p_{s}(x_{s}) &=&t_{s}(a_{1},\ldots ,t^{\mathbf{A}_{i}}(a_{r}\ldots
,a_{r^{\prime }}),\ldots ,a_{m}) \\
p_{s}(x_{s}^{\prime }) &=&t_{s}(a_{1},\ldots ,t(a_{r},\ldots ,a_{r^{\prime
}}),\ldots ,a_{m}) \\
p_{s+t}(x_{s+t}) &=&t_{s}(a_{1}^{\prime },\ldots ,t(a_{r}^{\prime },\ldots
,a_{r^{\prime }}^{\prime }),\ldots ,a_{m}^{\prime }) \\
p_{s+t}(x_{s+t}^{\prime }) &=&t_{s}(a_{1}^{\prime },\ldots ,t^{\mathbf{A}%
_{j}}(a_{r}^{\prime },\ldots ,a_{r^{\prime }}^{\prime }),\ldots
,a_{m}^{\prime })
\end{eqnarray*}%
where $i,j\in \{1,2\}$ and where%
\begin{eqnarray*}
x_{s} &=&t^{\mathbf{A}_{i}}(a_{r}\ldots ,a_{r^{\prime }}) \\
x_{s}^{\prime } &=&t(a_{r},\ldots ,a_{r^{\prime }}) \\
x_{s+t} &=&t(a_{r}^{\prime },\ldots ,a_{r^{\prime }}^{\prime }) \\
x_{s+t}^{\prime } &=&t^{\mathbf{A}_{j}}(a_{r}^{\prime },\ldots
,a_{r^{\prime }}^{\prime })\text{.}
\end{eqnarray*}%
Now, if $i=j$ then $a_{l}=a_{l}^{\prime }$, $r\leq l\leq r^{\prime }$, and
we can write a sequence%
\begin{equation*}
p_{s}(x_{s})\overset{\mathsf{H}^{\prime }\cup \mathsf{H}^{\prime -1}\cup 
\mathsf{I}_{A_{1}\cup A_{2}}}{\rightrightarrows }p_{s+t}(x_{s+t}^{\prime })%
\text{.}
\end{equation*}%
On the contrary, if $i\not= j$, then we may assume without losing generality that $i=1$ and $j=2$. 
Then $a_{l}=\phi _{1}(c_{l})$ and $a_{l}^{\prime }=\phi _{2}(c_{l})$
, for all $r\leq l\leq r^{\prime }$. This implies that 
\[t^{\mathbf{A}_{2}}(a_{r}^{\prime },\ldots
,a_{r^{\prime }}^{\prime })=
t^{\mathbf{A}_{2}}(\phi _{2}(c_{r}),\ldots ,\phi_{2}(c_{r^{\prime }}))=\phi
_{2}t^{\mathbf{C}}(c_{r},\ldots ,c_{r^{\prime }})
\]
\[t^{\mathbf{A}_{1}}(a_{r},\ldots ,a_{r^{\prime }})=
t^{\mathbf{A}_{1}}(\phi _{1}(c_{r}),\ldots ,\phi_{1}(c_{r^{\prime }}))=\phi
_{1}t^{\mathbf{C}}(c_{r},\ldots ,c_{r^{\prime }})
\]
and we can again write%
\begin{equation*}
p_{s}(x_{s})\overset{\mathsf{H}^{\prime }\cup \mathsf{H}^{\prime -1}\cup 
\mathsf{I}_{A_{1}\cup A_{2}}}{\rightrightarrows }p_{s+t}(x_{s+t}^{\prime })%
\text{.}
\end{equation*}

\noindent \textbf{Case 2c. }If in (\ref{segment1}) $x_{s}^{\prime }$ covers $%
x_{s+t}$, then we have the following subcases.\medskip

\noindent \textbf{(i)} Let $x_{s}$ and $x_{s+t}^{\prime }$ be both in $A_{i}$%
, $i\in \{1,2\}$. Assume that
\begin{eqnarray*}
p_{s}(x_{s}) &=&t_{s}(a_{1},\ldots ,t^{\mathbf{A}_{i}}(a_{r},\ldots
,t^{\prime }(a_{k},\ldots ,a_{k^{\prime }}),\ldots ,a_{r^{\prime }}),\ldots
,a_{m}) \\
p_{s}(x_{s}^{\prime }) &=&t_{s}(a_{1},\ldots ,t(a_{r},\ldots ,t^{\prime
}(a_{k},\ldots ,a_{k^{\prime }}),\ldots ,a_{r^{\prime }}),\ldots ,a_{m}) \\
p_{s+t}(x_{s+t}) &=&t_{s}(a_{1}^{\prime },\ldots ,a_{r}^{\prime },\ldots
,t^{\prime }(a_{k}^{\prime},\ldots ,a_{k^{\prime }}^{\prime}),\ldots ,a_{r^{\prime }}^{\prime },\ldots ,a_{m}^{\prime }) \\
p_{s+t}(x_{s+t}^{\prime }) &=&t_{s}(a_{1}^{\prime },\ldots ,t^{\prime \mathbf{A}_{i}}(a_{k}^{\prime},\ldots ,a_{k^{\prime }}^{\prime}),\ldots ,a_{m}^{\prime })
\end{eqnarray*}%
where $i\in \{1,2\}$, $(a_{\alpha },a_{\alpha }^{\prime })\in \mathsf{H}%
^{\prime }\cup \mathsf{H}^{\prime -1}\cup \mathsf{I}_{A_{1}\cup A_{2}}$, $%
1\leq \alpha \leq m$, and%
\begin{eqnarray*}
x_{s} &=&t^{\mathbf{A}_{i}}(a_{r},\ldots ,t^{\prime }(a_{k},\ldots
,a_{k^{\prime }}),\ldots ,a_{r^{\prime }})\in A_{i} \\
x_{s}^{\prime } &=&t(a_{r},\ldots ,t^{\prime }(a_{k},\ldots ,a_{k^{\prime
}}),\ldots ,a_{r^{\prime }})\in T(A_{i}) \\
x_{s+t} &=&t^{\prime }(a_{k}^{\prime },\ldots ,a_{k^{\prime }}^{\prime })\in
T(A_{i}) \\
x_{s+t}^{\prime } &=&t^{\prime \mathbf{A}_{i}}(a_{k}^{\prime },\ldots
,a_{k^{\prime }}^{\prime })\in A_{i}\text{.}
\end{eqnarray*}%
Letting%
\begin{eqnarray*}
p_{s}(x_{s}) &=&t_{s}(a_{1},\ldots ,t^{\mathbf{A}_{i}}(a_{r},\ldots
,t^{\prime \mathbf{A}_{i}}(a_{k},\ldots ,a_{k^{\prime }}),\ldots
,a_{r^{\prime }}),\ldots ,a_{m}) \\
p_{s}^{\prime }(x_{s}) &=&t_{s}(a_{1}^{\prime },\ldots ,a_{r-1}^{\prime
},t^{\mathbf{A}_{i}}(a_{r},\ldots ,t^{\prime \mathbf{A}_{i}}(a_{k},\ldots
,a_{k^{\prime }}),\ldots ,a_{r^{\prime }}),a_{r^{\prime }+1}^{\prime
},\ldots ,a_{m}^{\prime }) \\
p_{s}^{\prime }(x_{s}^{\prime }) &=&t_{s}(a_{1}^{\prime },\ldots
,a_{r-1}^{\prime },t(a_{r},\ldots ,t^{\prime \mathbf{A}_{i}}(a_{k},\ldots
,a_{k^{\prime }}),\ldots ,a_{r^{\prime }}),a_{r^{\prime }+1}^{\prime
},\ldots ,a_{m}^{\prime }) \\
p_{s+t}(x_{s+t}^{\prime }) &=&t_{s}(a_{1}^{\prime },\ldots ,a_{r}^{\prime
},\ldots ,a_{k-1}^{\prime },t^{\prime \mathbf{A}_{i}}(a_{k},\ldots
,a_{k^{\prime }}),a_{k^{\prime }+1}^{\prime },\ldots ,a_{r^{\prime
}}^{\prime },\ldots ,a_{m}^{\prime })
\end{eqnarray*}%
one may write:%
\begin{equation*}
p_{s}(x_{s})\overset{\mathsf{H}^{\prime }\cup \mathsf{H}^{\prime -1}\cup 
\mathsf{I}_{A_{1}\cup A_{2}}}{\rightrightarrows }p_{s}^{\prime }(x_{s})%
\overset{\mathsf{H}^{-1}}{\longrightarrow }p_{s}^{\prime }(x_{s}^{\prime })%
\overset{\mathsf{H}^{\prime }\cup \mathsf{H}^{\prime -1}\cup \mathsf{I}%
_{A_{1}\cup A_{2}}}{\rightrightarrows }p_{s+t}(x_{s+t}^{\prime })\text{.}
\end{equation*}

\noindent \textbf{(ii)} If $x_{s}$ and $x_{s+t}^{\prime }$ are not in the
same $A_i$, $i\in \{1, 2\}$, then we may assume without loss of generality that $%
x_{s}\in A_{1}$ and $x_{s+t}^{\prime }\in A_{2}$. Let us suppose that in (%
\ref{segment1}),%
\begin{eqnarray*}
p_{s}(x_{s}) &=&t_{s}(a_{1},\ldots ,t^{\mathbf{A}_{1}}(a_{r},\ldots
,t^{\prime }(a_{k},\ldots ,a_{k^{\prime }}),\ldots ,a_{r^{\prime }}),\ldots
,a_{m}) \\
p_{s}(x_{s}^{\prime }) &=&t_{s}(a_{1},\ldots ,t(a_{r},\ldots ,t^{\prime
}(a_{k},\ldots ,a_{k^{\prime }}),\ldots ,a_{r^{\prime }}),\ldots ,a_{m}) \\
p_{s+t}(x_{s+t}) &=&t_{s}(a_{1}^{\prime },\ldots ,t(a_{r}^{\prime },\ldots
,t^{\prime }(a_{k}^{\prime },\ldots ,a_{k^{\prime }}^{\prime }),\ldots
,a_{r^{\prime }}^{\prime }),\ldots ,a_{m}^{\prime }) \\
p_{s+t}(x_{s+t}^{\prime }) &=&t_{s}(a_{1}^{\prime },\ldots ,t(a_{r}^{\prime
},\ldots ,t^{\prime \mathbf{A}_{2}}(a_{k}^{\prime },\ldots ,a_{k^{\prime
}}^{\prime }),\ldots ,a_{r^{\prime }}^{\prime }),\ldots ,a_{m}^{\prime })
\end{eqnarray*}%
\noindent where $(a_{\alpha },a_{\alpha }^{\prime })\in 
\mathsf{H}^{\prime }\cup \mathsf{H}^{\prime -1}\cup \mathsf{I}_{A_{1}\cup
A_{2}}$, $1\leq \alpha \leq m$, and%
\begin{eqnarray*}
x_{s} &=&t^{\mathbf{A}_{1}}(a_{r},\ldots ,t^{\prime }(a_{k},\ldots
,a_{k^{\prime }}),\ldots ,a_{r^{\prime }})\in A_{1} \\
x_{s}^{\prime } &=&t(a_{r},\ldots ,t^{\prime }(a_{k},\ldots ,a_{k^{\prime
}}),\ldots ,a_{r^{\prime }})\in T(A_{1}) \\
x_{s+t} &=&t^{\prime }(a_{k}^{\prime },\ldots ,a_{k^{\prime }}^{\prime })\in
T(A_{2}) \\
x_{s+t}^{\prime } &=&t^{\prime \mathbf{A}_{2}}(a_{k}^{\prime },\ldots
,a_{k^{\prime }}^{\prime })\in A_{2}\text{.}
\end{eqnarray*}%
Again, as in Case 2b, we have $a_{l}=\phi _{1}(c_{l})$ and $a_{l}^{\prime }=\phi _{2}(c_{l})$
, for all $k\leq l\leq k^{\prime }$. This implies that 
\[t^{\prime\mathbf{A}_{2}}(a_{k}^{\prime },\ldots
,a_{k^{\prime }}^{\prime })=
t^{\prime\mathbf{A}_{2}}(\phi _{2}(c_{k}),\ldots ,\phi_{2}(k_{r^{\prime }}))=\phi
_{2}t^{\prime\mathbf{C}}(c_{k},\ldots ,c_{k^{\prime }})
\]
\[t^{\prime\mathbf{A}_{1}}(a_{k},\ldots ,a_{k^{\prime }})=
t^{\prime\mathbf{A}_{1}}(\phi _{1}(c_{k}),\ldots ,\phi_{1}(c_{k^{\prime }}))=
\phi_{1}t^{\prime\mathbf{C}}(c_{k},\ldots ,c_{k^{\prime }})
\]
and therefore

\begin{eqnarray*}
p_{s}(x_{s}) &=&t_{s}(a_{1},\ldots ,t^{\mathbf{A}_{1}}(a_{r},\ldots ,
\phi_{1}t^{\prime\mathbf{C}}(c_{k},\ldots ,c_{k^{\prime }})
,\ldots ,a_{r^{\prime }}),\ldots ,a_{m})
\\
p_{s}^{\prime }(x_{s}^{\prime }) &=&t_{s}(a_{1},\ldots
,t(a_{r},\ldots ,
\phi_{1}t^{\prime\mathbf{C}}(c_{k},\ldots ,c_{k^{\prime }}),
\ldots ,a_{r^{\prime }}),\ldots ,a_{m}) \\
p_{s+t}(x_{s+t}^{\prime }) &=&t_{s}(a_{1}^{\prime },\ldots ,t(a_{r}^{\prime
},\ldots ,
\phi
_{2}t^{\prime\mathbf{C}}(c_{k},\ldots ,c_{k^{\prime }}),
\ldots
,a_{r^{\prime }}^{\prime }),\ldots ,a_{m}^{\prime })
\end{eqnarray*}

This allows us to write%
\begin{equation*}
p_{s}(x_{s})
\overset{\mathsf{H}^{-1}}{\longrightarrow }p_{s}^{\prime }(x_{s}^{\prime })%
\overset{\mathsf{H}^{\prime }\cup \mathsf{H}^{\prime -1}\cup \mathsf{I}%
_{A_{1}\cup A_{2}}}{\rightrightarrows }p_{s+t}(x_{s+t}^{\prime })\text{.}
\end{equation*}

\noindent \textbf{Case 2d.} When $x_{s+t}$ covers $x_{s}^{\prime }$. This case can be handled in a similar way as Case 2c. 

Iterating the procedures mentioned in different subcases of Case 2 we can remove all occurrences of $H$ and $H^{-1}$ in (\ref{Sequence}). This implies that Case 2 can be reduced to Case 1.
Hence, the theorem is proved. 
\end{proof}

\begin{corollary}
Epis are surjective in $\mathcal{F}$-\textbf{Oalg}$_{\leq }^{0}$.
\end{corollary}
\begin{proof}
Follows from Theorem \ref{main} and Corollary \ref{Cor2}. 
\end{proof}

\begin{corollary}
Epis are surjective in $\mathcal{F}$-\textbf{Alg}, the variety of all unordered algebras of type $\mathcal{F}$.
\end{corollary}
\begin{proof}
Let $f$ be an epi in $\mathcal{F}$-\textbf{Alg}. Then, $f$ is also an epi in $\mathcal{F}$-\textbf{Oalg}$_{= }^{0}$.
This implies, by Theorem \ref{main}, that  $f$ is surjective. 
\end{proof}

\end{document}